\documentclass[10pt]{article}
\usepackage[utf8]{inputenc}
\usepackage[usenames,dvipsnames,svgnames,table]{xcolor}	
\usepackage[]{amsmath}
\usepackage{cases}
\usepackage{empheq}
\usepackage[colorlinks=true]{hyperref}
\hypersetup{colorlinks,
citecolor=brown, linkcolor=blue, linktocpage
}
\usepackage{cleveref} 
\usepackage{amssymb}
\usepackage{bbm}
\usepackage{enumerate}   
\usepackage{geometry}
\usepackage{amsfonts}
\usepackage{color}
\usepackage[]{amsthm} 
 \usepackage[shortlabels]{enumitem}

\newtheorem{theorem}{Theorem}[section]
\newtheorem{lemma}[theorem]{Lemma}
\newtheorem{cor}[theorem]{Corollary}
\newtheorem{proposition}[theorem]{Proposition}
\newtheorem{mydef}[theorem]{Definition}

\newtheorem{remark}[theorem]{Remark}

\def\E{\mathbb{E}}
\def\R{\mathbb{R}}
\def\P{\mathbb{P}}
\def\Q{\mathbb{Q}}
\def\N{\mathbb{N}}

\def\1{\mathbbm{1}}

\def\D1{\frac{\partial}{\partial x}}

\usepackage{xspace}

\title{On a multi-dimensional McKean-Vlasov SDE with memorial and singular interaction associated to the parabolic-parabolic Keller-Segel model }
\author{Milica Toma\v{s}evi\'c\footnote{ CMAP, CNRS, Ecole Polytechnique, Institut Polytechnique de Paris, 
91128 Palaiseau, France; milica.tomasevic@polytechnique.edu.} , Guillaume Woessner \footnote{ PROB, MATH, SB, EPFL, CH-1015 Lausanne, Switzerland; guillaume.woessner@epfl.ch.}}
\date{}

\begin{document}
\maketitle
\noindent
\textbf{Abstract:} In this work we firstly prove the well-posedness of the non-linear martingale problem related to a McKean-Vlasov stochastic differential equation with singular interaction kernel in $\R^d$ for $d\geq 3$.  The particularity of our setting is that the McKean-Vlasov process we study interacts at each time  with all its past time marginal laws by means of a singular space-time kernel.\\
Secondly, we prove that our stochastic process is a probabilistic interpretation for the parabolic-parabolic Keller-Segel system in $\R^d$. We thus obtain a well-posedness result to the latter under explicit smallness condition on the parameters of the model. 
\textbf{Classification:} 60H30, 60H10.

\section{Introduction and results}
\subsection{The problem}
The goal of this work is to obtain, in the multi-dimensional framework ($d\geq 3$), a well-posedness result for the following non-linear McKean-Vlasov stochastic differential equation (SDE): 
\begin{equation}
\label{EDS_MKV}
\begin{cases}
&dX_t = \chi \, b_0(t,X_t)dt +  \chi\left( \int_0^t  K_{t-s}\ast p_s(X_t)\,ds\right)dt + dW_t, \quad t>0,\\
& \mathcal{L}(X_t)= p_t, \ t >0, \quad \mathcal{L}(X_0)=\rho_0,\\
\end{cases}
\end{equation}
and to establish its relationship with a system of two parabolic PDEs coupled in a non-linear way.\\
Here $\chi>0$ and $W$ is a $d$-dimensional standard Brownian motion on a filtered probability space $(\Omega, \mathcal{F}, \P)$. The linear part of the drift $b_0$ belongs to $L^1_{loc}(\R_+;L^r(\R^d))$ for $r\geq 2$ and will be made explicit  below.

The main difficulty when dealing with  \eqref{EDS_MKV} is the unusual interaction the process $X$ has with its law. At each time $t>0$, $X_t$ interacts with all its past marginal laws up to $t$ by means of the following singular time-space kernel $K: \R_+ \times \R^d \to \R^d $ 
$$K_t(x):=-\dfrac{x}{(2\pi t)^{d/2}t}e^{-\frac{|x|^2}{2t}-\lambda t},$$
where $\lambda\geq 0$. Notice that K is not well defined when $(t,x)\to (0,0)$ and that the time integral in \eqref{EDS_MKV} may not be well defined as one should integrate around $t$ a time function of the order $(t-s)^{-\frac{d+1}{2}} ds$.

The above interaction may look surprising, but it is actually a consequence of a rather physical (more precisely, biological)  situation. Namely, we will prove here that the above process is a probabilistic interpretation of a PDE. Namely, it will be established that $X$ is   a typical particle in an infinite system of particles undergoing the dynamics of a Fokker-Planck equation which is, in the drift term, coupled with another parabolic equation. This system reads:

\begin{equation}
\label{EDP_KS}
\begin{cases}
& \partial_t \rho(t,x) = \frac{1}{2}\Delta  \rho(t,x)-\chi \nabla \cdot \left(\rho(t,x)\nabla c(t,x)\right), \qquad t>0,\quad x\in\R^d \\
& \partial_t c(t,x) = \frac{1}{2}\Delta c(t,x)-\lambda c(t,x)+\rho(t,x), \qquad t>0, \quad x\in\R^d, \\
& \rho(0,x)=\rho_0(x), \quad c(0,x)=c_0(x), \qquad x\in\R^d,
\end{cases}
\end{equation}
and it is the well-known parabolic-parabolic Keller-Segel minimal model \cite{KELLER_SEGEL1, KELLER_SEGEL2, KELLER_SEGEL3}. It describes the directed movement of a population (like bacteria or cells) triggered by  the presence of an attractive substance (the so-called \textit{chemo-attractant}). This phenomenon is called \textit{chemotaxis}. The functions $\rho$ and $c$ indicate respectively the population density and the chemo-attractant concentration in time and space. The cells diffuse and follow the gradient of the concentration of the chemical substance with the \textit{chemo-attractant sensitivity} $\chi$, while the chemical also diffuses, decays with rate $\lambda\geq 0$ and it is produced by the cell population with rate $1$. \\
The Keller-Segel system has been extensively studied in the past 40 years and is still a very active topic. One of its most important features is that the solutions may blow up in finite time, even though the total mass of the equation \eqref{EDP_KS} is preserved. This happens depending on the dimension $d$ and the size of parameter $\chi$ (see Horstmann \cite{Horstmann1} for a complete review). In particular, for the parabolic-elliptic version of \eqref{EDP_KS} blow-up in finite time may occur for $d=2$ if $\chi >8\pi$ and for $d\geq 3$ if $\frac{1}{2}\int |x|^2 u_0(dx)\leq  C(\chi)$ for some fixed $ C(\chi)>0$ (for a precise statement see \cite[Thm.~2.2]{perthame}). In the parabolic-parabolic setting, the solutions are global when the total mass of the initial condition is small enough, but above the critical value there exists solutions that blow up in finite time along other that are global (see \cite{Mizo}).  Our objective here is to derive a probabilistic framework for studying the doubly parabolic system in higher dimensions. As a first step, we will prove the well-posedness of \eqref{EDS_MKV} and its link to \eqref{EDP_KS}. The next steps are discussed in the perspectives.
\newline
~\\
In order to see \eqref{EDS_MKV} as a probabilistic interpretation of \eqref{EDP_KS} one should adopt formally the following decoupling strategy. Write the Duhamel formula for the concentration equation,
\begin{equation}
\label{DUH_c}
c(t,x)=e^{-\lambda t}(g_t\ast c_0) (x) + \int_0^t e^{-\lambda s} (\rho_{t-s} \ast g_s)(x)ds,
\end{equation} 
and plug it in the density equation by putting the gradient on the Gaussian density $g_t(x):= \frac{1}{(2\pi t)^{d/2}}e^{-\frac{|x|^2}{2t}}$. Then, writing the stochastic process corresponding to the latter, one obtains \eqref{EDS_MKV}.  This procedure was first done by Talay and Toma\v sevi\'c~\cite{Mi-De} in the one-dimensional framework and is explained in detail in Toma\v sevi\'c~\cite{tomasevic.these}. In order to rigorously justify it, one should first construct the process  \eqref{EDS_MKV} and extract the family of one dimensional time marginal laws of $X$, $(\rho^X_t)_{t\geq 0}$. Second, one should define $c^X$ as a transformation of the family $ \rho^X$ given in \eqref{DUH_c}. Finally, it should be proven that the couple $(\rho^X,c^X)$ is a solution to \eqref{EDP_KS} in some suitable sense.

In \cite{Mi-De}, respectively Toma\v sevi\'c~\cite{tomasevic.2}, the authors find conditions on the parameters and/or on the initial data, under which the existence and uniqueness of a martingale problem related to \eqref{EDS_MKV} holds in dimension~$1$, respectively in dimension~$2$ and they perform the above program. In this work, we will tackle the case of $d\geq 3$. 
\subsection{Multi-dimensional setting vs $d=1,2$}
\label{ss_multi_d}
 The main ingredient when proving existence to \eqref{EDS_MKV} is that one needs to construct a solution in a suitable probability space in which the one dimensional time marginals of the process belong to functional spaces that are able to tame the singularity of the kernel through the convolution in space. In this section we explain how the multi-dimensional setting $d\geq 3$ compares to the cases $d=1$ and $d=2$ and we  give a heuristic reasoning that leads us to the appropriate functional setting. 

It is worth of noticing that the space-time singularity of the kernel $K$ increases with the increase of the spatial dimension $d$. As the kernel is convoluted in time and space with the density of the process, we need more and more regularity on the latter to handle the singularity of the kernel in the non-linear part of the drift of the process. One way to see this is to check the mixed $L^q_{loc}(L^p(\R^d))$-regularity of $K$ for $p,q\geq 1$. As $K$ behaves as the spatial derivative of the heat kernel, it obviously belongs to  $L^1_{loc}(L^1(\R^d))$ for any $d\geq 1$. However, it belongs to $L^1_{loc}(L^p(\R^d))$ for $p\in[1, \frac{d}{d-1})$ which is more and more restrictive as $d$ increases. \\
This is in complete contrast with the case of dimension $1$ where $K\in L^1_{loc}(L^p(\R))$ for any $p\geq 1$.  That is why in \cite{Mi-De}, global in time solutions to \eqref{EDS_MKV}
 were constructed for $p_0\in L^1(\R)$ and without any condition on the parameters of the model. The authors worked with probability spaces where the time marginal laws of the non-linear process had the following behaviour: $\sup_{t\leq T}\sqrt{t}\|p_t\|_\infty<\infty $. Thus, the corresponding drift was uniformly bounded in time and space. This was achieved by Picard's iterations and simultaneous controls on the drift and density terms in each iteration.\\
 In dimension $2$ such strategy was no longer possible and was replaced by a regularization procedure and the drift-density controls of the regularised process in suitable spaces. To obtain global existence, the initial condition in \cite{tomasevic.2} had the same regularity as in $d=1$ case, but there was a restriction on the parameter $\chi$ to account for the increased singularity of the kernel. The author worked with time marginal laws that essentially had Gaussian behaviour, i.e.  $\sup_{t\leq T}t^{1-\frac{1}{p}}\|p_t\|_p<\infty $ for any $p>1$. This lead to a drift term $(b_t)_{t\leq T}$ whose regularity was $\sup_{t\leq T}t^{\frac{1}{2}-\frac{1}{r}}\|b_t\|_r<\infty $, for $r\in [2,\infty]$.\\
 In the framework of $d\geq 3$, we adopt the strategy of the two-dimensional framework and we expect again a smallness  condition on the parameter $\chi$. However, as we will explain informally in what follows,   an additional condition for the initial condition $p_0$ is also necessary. 
 
 To find the suitable functional framework in $d\geq 3$ case, let us analyse a-priori the following mild equation for the one dimensional time marginals of \eqref{EDS_MKV}: 
 \begin{equation}
     \label{eq:mild-apriori}
     p_t = g_t\ast p_0 -  \sum_{i=1}^d \int_0^t \nabla_i g_{t-s}\ast (p_s b_s^i )ds,
 \end{equation}
 where $b_t$ denotes the sum of the linear and non-linear drift in \eqref{EDS_MKV} at time $t>0$ and $b_t^i$ its $i$-th component. Naturally, we first proceed as in $d=2$ case and 
assume only that $p_0\in L^1(\R^d)$. Then, the term $ g_t\ast p_0$ will have the Gaussian regularity in $\R^d$, that is for any $q>1$ and $t>0$, 
\begin{equation}
    \label{eq:regularity}
     \|g_t\ast p_0\|_q\leq C t^{-\frac{d}{2}(1-\frac{1}{q})}.
\end{equation}
We cannot expect a better behaviour  for $\|p_t\|_q$ while $t\to 0$ and we certainly do not wish for it to be worse (as the drift is the convolution of our singular kernel with the density). Hence, we search for the a-priori estimates on $p_t$ of the form $ \|p_t\|_p\leq C t^{-\frac{d}{2}(1-\frac{1}{q})}$ for a suitable $q$.

Supposing the latter, we will first derive the behaviour of the drift term and than use it  to come back to \eqref{eq:mild-apriori} and see if the regularity of $p_t$ is preserved. Let $p_0 \in L^1(\R^d)$, $c_0$ in a suitable space and $\sup_{t\leq T}t^{\frac{d}{2}(1-\frac{1}{q})}\|p_t\|_q<\infty $ for some $q$ to be fixed. Then, after a convolution inequality one has that $$\|b_t\|_r<\|b_0(t,\cdot)\|_r + C \int_0^t \frac{1}{(t-s)^{\frac{d}{2}(\frac{1}{q}-\frac{1}{r})+\frac{1}{2}} s^{\frac{d}{2}(1-\frac{1}{q})}}  ds \leq \frac{C}{t^{\frac{d}{2}(1-\frac{1}{r})-\frac{1}{2}}},$$ for suitable $c_0$,  $q< \frac{d}{d-2}$  and $\frac{1}{r}>\frac{1}{q}-\frac{1}{d}$ (which gives with the aforementioned condition on $q$ that $r<\frac{d}{d-3}$). Now, having the drift and density regularity in mind the idea is to treat the second term in the r.h.s. of \eqref{eq:mild-apriori} with a convolution and Cauchy-Schwarz inequality successively. Namely, denoting by $q'$ the conjugate of $q$, notice that $1+\frac{1}{q}= \frac{1}{q'}+\frac{q}{p}$,
\begin{equation}
\label{eq:mild-apriori2}
\sum_{i=1}^d \int_0^t \|\nabla_i g_{t-s}\ast (p_s b_s^i )\|_q ds \leq C_d  \int_0^t \frac{1}{(t-s)^{\frac{d}{2}(1-\frac{1}{q'})+\frac{1}{2}} s^{\frac{d}{2}(1-\frac{1}{q})} s^{\frac{d}{2}(1-\frac{1}{q})-\frac{1}{2}}}ds.
\end{equation}
Assume we are able to integrate, then the term on the r.h.s. of \eqref{eq:mild-apriori2} would have the regularity $t^{-\frac{d}{2}(2-\frac{1}{q}) +1}$ (for more details see \eqref{6.3.2}). This will  be the desired $t^{-\frac{d}{2}(1-\frac{1}{q})}$ if and only if $d=2$. Hence, the case $d=2$ is somehow a turning point as it is the one that introduces the parameter condition without imposing more regularity on the initial condition. In fact, it will not even be possible to integrate the r.h.s. of \eqref{eq:mild-apriori2} whatever the value of $q$ is. Indeed, we would need that simultaneously  $d(1-\frac{1}{q})-\frac{1}{2}<1$ and $\frac{d}{2}(1-\frac{1}{q'})+\frac{1}{2}<1$ which becomes 
$q<\frac{2d}{2d-3}$  and $q>d$.
This is not possible for $d\geq 3$ (while for $d=2$ this condition is $q\in (2,4)$). Hence the above reasoning cannot work if we only have $p_0\in L^1(\R^d)$.

In the next few lines, we will see  how, imposing more regularity on the initial condition, the above framework changes.  Let us assume that $p_0\in L^1\cap L^{q_0}(\R^d)$ for some $q_0>1$ that we will fix below. First, we see that after a convolution inequality one has for some $q\geq 1$,
$$ \|g_t\ast p_0\|_q\leq C \|p_0\|_{q_0} t^{-\frac{d}{2}(\frac{1}{q_0}-\frac{1}{q})}.$$
 As before,  we will  search for the a-priori estimates on $p_t$ of the form $ \|p_t\|_q\leq C t^{-\frac{d}{2}(\frac{1}{q_0}-\frac{1}{q})}$ for a suitable $q$ and $q_0$.
Of course, as this explodes slower around zero than \eqref{eq:regularity}, we expect that this will give us enough space to make work the above arguments.\\
As a second step, we check how the regularity of the drift changes. For suitable values of $r\geq 1$ (admitting that one may integrate), it comes
$$\|b_t\|_r<\|b_0(t,\cdot)\|_r + C \int_0^t \frac{1}{(t-s)^{\frac{d}{2}(\frac{1}{q}-\frac{1}{r})+\frac{1}{2}} s^{\frac{d}{2}(\frac{1}{q_0}-\frac{1}{q})}}  ds \leq \frac{C}{t^{\frac{d}{2}(\frac{1}{q_0}-\frac{1}{r})-\frac{1}{2}}}.$$ 
With this regularity, \eqref{eq:mild-apriori2} transforms into 
\begin{equation}
\label{eq:mild-apriori3}
\sum_{i=1}^d \int_0^t \|\nabla_i g_{t-s}\ast (p_s b_s^i )\|_q ds \leq C_d  \int_0^t \frac{1}{(t-s)^{\frac{d}{2}(1-\frac{1}{q'})+\frac{1}{2}} s^{\frac{d}{2}(\frac{1}{q_0}-\frac{1}{q})} s^{\frac{d}{2}(\frac{1}{q_0}-\frac{1}{q})-\frac{1}{2}}}ds.
\end{equation}
Assuming once again that we can integrate, the term on the r.h.s. of \eqref{eq:mild-apriori3} would have the regularity $t^{-\frac{d}{2}(\frac{2}{q_0}-\frac{1}{q}) +1}$.  This will be the desired $t^{-\frac{d}{2}(\frac{1}{q_0}-\frac{1}{q})}$ if and only if $q_0=\frac{d}{2}$. In addition, fixing the latter integrability condition will come down to $q\in (d,2d)$. 

The conclusion is that, contrary to the cases $d=1,2$, in higher dimensions the Gaussian regularity does not suffice to integrate our singular kernel. However, we can still work in $L^q$ spaces by increasing the regularity of the initial condition. The additional regularity is the $L^{\frac{d}{2}}$ one, while the space in which we work is $\sup_{t\leq T} t^{1-\frac{d}{2q}}\|p_t\|_q<\infty$, for $q\in (d,2d)$.  The method we employ generalizes the method developed in \cite{tomasevic.2} with the above difference in the functional setting. \\
Another important difference is the following. In \cite{tomasevic.2} the functional setting we mentioned worked for any $q\geq 1$. As we have seen above, first there was a condition that $q\in (2,4)$, but actually after obtaining the correct estimates for such $q$, one can generalize to any $\tilde{q}\in (1,q)$ by interpolation and any $\tilde{q}>q$ using parabolic regularity. Here, the situation is more involved. The first step passes using the above computations for $q\in (d,2d)$. Also, the parabolic regularity does the job for $\tilde{q}>q$. However, a simple interpolation between $1$ and $q$ does not give the correct time decay. This is because we have imposed more regularity on the initial condition and actually there should not be any time decay for $\tilde{q}\in (1, \frac{d}{2})$. Hence, we need to apply an iterative procedure applying successively interpolation and parabolic regularity in order to obtain that $\sup_{t\leq T}\|p_t\|_{\tilde{q}}<\infty$ for $\tilde{q}\in (1, \frac{d}{2}]$. Then, interpolating between $\frac{d}{2}$ and $q$ gives the correct time decay and our regularity is valid for all values of $\tilde{q}>1$. That is why we have chosen to solve the martingale problem for a fixed $q\in (d,2d)$ and then extend the estimates separately in Section \ref{S_properties}. Hence, we have less restriction on the probability space we need to work in in order to solve the martingale problem than in \cite{tomasevic.2}.

The reader is reminded that the above discussion is purely heuristic and that the goal was just to explain the additional difficulty we met by considering the $d\geq 3$ framework. The above calculations and remarks will be rigorously  detailed when we will analyse a regularised version of \eqref{eq:mild-apriori} in Section~\ref{S_Preliminaries} and when we will obtain further properties of our solutions in Section \ref{S_properties}.

\subsection{Notations and definitions}

\begin{itemize}
    \item[$\bullet$] Let $0<a,b<1$, we denote the so-called $beta$-integral as
\begin{align}
\beta(a,b):=\int_0^1 \dfrac{1}{u^a(1-u)^b}du.
\end{align}
The change of variables $\frac{s}{t}=u$ leads to the following frequently used identity, for $t>0$:
\begin{align}
\label{6.3.2}
\int_0^t \dfrac{1}{s^a(t-s)^b}ds=t^{1-(a+b)}\beta(a,b).
\end{align}
Moreover, one can show that for any $\varepsilon>0$
\begin{align}
\label{6.3.2.1}
\beta_\varepsilon := \sup \{\beta(a,b)~:~a,b \leqslant 1-\varepsilon\} <\infty.
\end{align}
    \item[$\bullet$] In this paper, $L^q(\R^d)$ for $q\geq 1$ will always denote the usual space of functions on $\R^d$ whose norm at the power $q$ is integrable, and the associated norm is denoted by $\|\cdot\|_q$. Moreover, $W^{1,q}(\R^d)$ is the related usual Sobolev space, whose norm will be denoted by $\|\cdot\|_{1,q}$.
    \item[$\bullet$] Throughout the paper we will note by $q'\geq 1$ the conjugate of a parameter $q\geq 1$ defined by the equality $1=\frac{1}{q}+\frac{1}{q'}$.
    \item[$\bullet$] For $(t,x)\in \R^\ast_+ \times \R^d$, we denote the Gaussian density by 
\begin{align}
\label{normale}
g_t(x):=\dfrac{1}{(2\pi t)^{d/2}}\exp\left\{-\dfrac{|x|^2}{2t}\right\}.
\end{align}
In this notation, our interaction kernel becomes $K_t(x)=e^{-\lambda t} \nabla g_t(x)$. We will frequently use throughout the whole paper the following facts about the Gaussian density.
For $1\leq r \leq \infty$, one has
$$\|g_t\|_r = \dfrac{C_0(r)}{t^{\frac{d}{2}(1-\frac{1}{r})}}.$$ Moreover, for all $i = 1,\dots,d$ one has,
$$\|\nabla_i g_t\|_r = \dfrac{C_1(r)}{t^{\frac{d}{2}(1-\frac{1}{r})+\frac{1}{2}}}.$$
The constants appearing above are made explicit in the appendix (see Lemma~\ref{lemme_normale}). 
    \item[$\bullet$] Throughout the paper the linear part of the drift is given by $b_0(t,x):=e^{-\lambda t} g_t\ast \nabla c_0$, where $c_0$ is the initial chemo-attractant concentration given in \eqref{EDP_KS}. 

For a probability measure $\Q$ on  $C([0,T],\R^d)$ we usually denote by $q_t$ its one-dimensional marginals when they are absolutely continuous w.r.t. Lebesgue measure, and we introduce the measure dependant map $b: \R_+ \times \R^d \times \mathcal{P}(C([0,T],\R^d)) \to \R^d$, assuming the the marginal densities $q_t$ indeed exist,
\begin{equation}
\label{def_b}
b(t,x,\Q):= \chi b_0(t,x) + \chi\left( \int_0^t K_{t-s}\ast q_s(x)ds\right)
\end{equation}
\item[$\bullet$] Finally, the notations $C$ or $C(\cdot)$ will denote throughout the paper a constant that does not depend on any relevant variable, and may change its value from one line to the other.
\end{itemize}

\subsection{Main results}
Let us first define the martingale problem related to \eqref{EDS_MKV}. As we work with weak solutions, we solve the non linear martingale problem related to (\ref{EDS_MKV}). By standard arguments, one passes from a solution to a (non-linear) martingale problem to a weak solution of a (non-linear) SDE \eqref{EDS_MKV}. The main issue when defining the martingale problem is choosing the right measure space for its solution. In other words, the probability measure on $\mathcal{C}([0,T];\R^d)$  that solves this problem must have some additional properties that enables us to tame the singularity of the interaction kernel. We have the following definition: 
\begin{mydef}
\label{defMP}
Let $T>0$, $\chi>0$, $d\geq 3$ and $q\in (d,2d)$. Consider the canonical space $\mathcal{C}([0,T];\R^d)$ equipped with its canonical filtration. Let $\P$ be a probability measure on this canonical space and denote for $t>0$ by $\P_t$ its one dimensional time marginals. $\P$ solves the non-linear 
martingale problem \hyperref[defMP]{$\mathcal{(MP)}$} if:
\begin{enumerate}[(i)]
\item $\P_0 = \rho_0$.
\item For any $t\in (0,T]$, $\P_t$ admit densities $p_t$ w.r.t. Lebesgue measure on $\R^d$. In addition, they satisfy 
\begin{equation}
\label{dens_Est_d}
\sup_{t\leq T}t^{1-\frac{d}{2q}}\|p_t\|_{L^q(\R^d)}\leq C_q(\chi),
\end{equation}
where $C_q(\chi)>0$.
\item For any $f \in C_K^2(\R^d)$ the process $(M_t)_{t\leq T}$, 
defined as
\begin{equation}
\label{def_mart}
M_t:=f(w_t)-f(w_0)-\int_0^t \Big[\frac{1}{2} \triangle f(w_u)+\nabla f(w_u)\cdot b(u, w_u, \P)  \Big]du
\end{equation}
is a $\P$-martingale where $(w_t)$ is the canonical process.  
\end{enumerate}
\end{mydef}
Notice that the condition ($ii$) is not usually required. Working with marginal laws that satisfy \textit{(ii)} ensures that the drift of \hyperref[defMP]{$\mathcal{(MP)}$} is integrable and so \hyperref[defMP]{$\mathcal{(MP)}$} is well defined. Indeed, for  $\nabla c_0 \in L^{d}(\R^d)$ and $\P$  as in \hyperref[defMP]{$\mathcal{(MP)}$}, one has for a fixed $u>0$, after applying  H\"older inequalities and Lemma \ref{lemme_normale},
$$\|b(u, \cdot, \P)\|_\infty \leq \|\nabla c\|_d \frac{C}{\sqrt{u}} + \int_0^u \|K_{u-s}\|_{q'}\|p_s\|_q ds\leq \frac{1}{\sqrt{u}}(\|\nabla c\|_d+ \beta(1-\frac{d}{2q},\frac{d}{2q} + \frac{1}{2} )).$$

Now, we may present our first main result about the existence of a solution to \hyperref[defMP]{$\mathcal{(MP)}$}.
\begin{theorem}
\label{6.2.3}
 Let $T>0, \lambda \geq 0$, $\chi>0$, $d\geq 3$ and fix an arbitrary $q\in (d,2d)$. Assume that $\rho_0$ is a probability density function that belongs to $L^{\frac{d}{2}}(\R^d)$, and suppose that $ c_0\in W^{1,d}(\R^d)$. Then \hyperref[defMP]{$\mathcal{(MP)}$} admits a solution under the following condition:
\begin{align}
\label{condition}
A\chi \|\nabla c_0\|_d + B\sqrt{\chi \|p_0\|_{\frac{d}{2}}} <1.
\end{align}
Here
\begin{align*}
& A:=dC_1(q')C_0\left(\tilde{q}_1\right)\beta\left(\frac{3}{2}-\frac{d}{q},\frac{d}{2}(1-\frac{1}{q'})+\frac{1}{2}\right),\\
&B:=2\sqrt{C_0(\tilde{q}_2)dC_1(q')C_1(1)\beta\left(\frac{3}{2}-\frac{d}{q},\frac{d}{2}(1-\frac{1}{q'})+\frac{1}{2}\right)\beta\left(1-\dfrac{d}{2q},\dfrac{1}{2}\right)},
\end{align*}
and $q'$ is such that $\frac{1}{q}+\frac{1}{q'}=1$, the functions $C_1(\cdot), C_2(\cdot)$ are defined in Lemma \ref{lemme_normale} and $\beta(\cdot,\cdot)$ in (\ref{6.3.6}), while $\tilde{q}_1=\frac{dq}{(d-1)q+d}$ and $\tilde{q}_2 = \frac{dq}{d+(d-2)q}$.
\end{theorem}

The technique we use to prove Theorem \ref{6.2.3} is the following. We start by conveniently regularizing the kernel $K$ and the linear drift $b_0$. The corresponding regularized process is well posed. Then, we extensively analyse its one dimensional marginal distributions in order to obtain a property of the type \eqref{dens_Est_d} independently of the regularization parameter. This is where the condition \eqref{condition} emerges and thanks to this we are able to obtain the tightness w.r.t the regularization parameter of the probability laws of the regularized process. Then, we prove that the limit point of this family of laws satisfies \hyperref[defMP]{$\mathcal{(MP)}$}.

So far we have constructed a weak solution to \eqref{EDS_MKV}. Our next goal is to validate \eqref{EDS_MKV} as a probabilistic interpretation of \eqref{EDP_KS}. With the following definition we  precise the notion of solutions to~\eqref{EDP_KS}.
\begin{mydef}
\label{6.2.4}
Given  $T>0, \lambda >0$, $\chi>0$ and  $q\in (d,2d)$, let $\rho_0$ be a probability density function that belongs to $L^{\frac{d}{2}}(\R^d)$, and let $ c_0\in W^{1,d}(\R^d)$. A pair $(\rho,c)$ is a solution of PDE (\ref{EDP_KS}) if for any $0<t \leq T$,  $\rho(t,\cdot)$ admits a density with respect to Lebesgue measure  such 
$$\sup_{t\leqslant T} t^{1-\frac{1}{q}} \|\rho(t,\cdot)\|_q \leqslant C_q,$$
for some  $C_q>0$. 
Moreover, the following equation is satisfied in the sense of the distributions,
\begin{align}
\label{6.10}
\rho(t,x)=g_t\ast\rho_0(x) - \chi \sum_{i=1}^d \int_0^t \nabla_i g_{t-s} \ast (\rho(s,\cdot)\nabla_i c(s,\cdot))(x)ds
\end{align}
with $$c(t,x) = e^{-\lambda t} (g(t,\cdot)\ast c_0)(x) + \int_0^t e^{-\lambda s}(g_s\ast \rho(t-s,\cdot))(x)ds,$$
\end{mydef}

Now, we are ready to state our corollary result of Theorem \ref{6.2.3} in which we construct a solution to \eqref{EDP_KS} from our non-linear process $X$.
\begin{cor}
\label{6.2.5}
Assume the hypothesis of Theorem~\ref{6.2.3}.  Denote  $\rho(t,\cdot): = \mathcal{L}(X_t)$, for $t\geq 0$ and define $c$ as a transformation of $\rho$ given by \eqref{DUH_c}. Then, this couple $(\rho,c)$ is a solution, in the sense of Definition \ref{6.2.4}, to (\ref{EDP_KS}).
\end{cor}

Besides, we were able to obtain results about the uniqueness to the solution to EDS \eqref{EDS_MKV} given in Theorem \ref{6.2.3} and to EDP \eqref{EDP_KS} given in Corollary \ref{6.2.5}. 

\begin{proposition}
\label{6.2.5.1}
Assume the hypothesis of Theorem~\ref{6.2.3}. Under the additional condition
\begin{align}
\label{condition_2}
\chi  C(\|\nabla c_0\|_d , C_q)<1,
\end{align} 
where 
$$C(\|\nabla c_0\|_d , C_q) := C_1(1) \beta\left(\frac{1}{2},\frac{1}{2}\right) \left[ \|\nabla c_0\|_d + C_q C_1(q')\left(\beta\left(\frac{d}{2q}+\frac{1}{2}, 1 - \frac{d}{2q}\right) + 1 \right) \right],$$
 Keller-Segel system \eqref{EDP_KS} admits the  unique solution in the sense of Definition \ref{6.2.4}.
\end{proposition}

\begin{theorem}
\label{6.2.5.2}
Let the assumptions of Proposition \ref{6.2.5.1} hold. Then, uniqueness holds for the martingale problem  \hyperref[defMP]{$\mathcal{(MP)}$}.
\end{theorem}

\begin{remark}
We will not explicitly prove Theorem \ref{6.2.5.2}, since the strategy to follow is exactly the same as in \cite{tomasevic.2}, but rather give a sketch. The linearization strategy used in \cite{tomasevic.2} is very common in the litterature about uniqueness of solutions to non-linear SDEs. More precisely, from the uniqueness given by Proposition \ref{6.2.5.1} of the solution $(\rho_s)_{s\in[0,T]}$ to PDE \eqref{EDP_KS} given by Corollary \ref{6.2.5}, one can define a linearized SDE as
\begin{equation}
\label{EDS_linearisee}
\begin{cases}
&d\tilde{X}_t = \chi \, b_0(t,\tilde{X}_t)dt +  \chi\left( \int_0^t  K_{t-s}\ast p_s(\tilde{X}_t)\,ds\right)dt + dW_t, \quad t>0,\\
& \mathcal{L}(X_0)=\rho_0,\\
\end{cases}
\end{equation}
Then, one should use the so-called "transfer of uniqueness" from \cite{Trevisan} from a linear PDE to a linear martingale problem. This procedure is explained in detail in \cite[p.19]{tomasevic.2}. Once the uniqueness of the linear martingale problem is obtained, since every solution to \eqref{EDS_MKV} is also a solution to \eqref{EDS_linearisee}, one can deduce the uniqueness in Proposition \ref{6.2.5.2}.
\end{remark}

\subsection{Literature and perspectives}
\paragraph{Litterature}~\\
We conclude this section by giving some context to the problem with respect to the literature. We focus on the references that have not been mentioned so far.
 
From the probability side, the study of non-linear Mc-Kean Vlasov processes with irregular coefficients has gained much attention recently, and its study can be split in two categories. On the one hand, singular interacting kernels are usually studied in a case by case basis applying original techniques, see for example Fournier and Jourdain \cite{fournier-jourdain} for a probabilistic approach to the parabolic-elliptic Keller-Segel model in dimension $2$, Osada \cite{Osada85} and Fournier, Hauray and Mischler in \cite{Fournier2014} for $2$-$d$ Navier Stokes equations in vorticity form with a repulsive kernel, or Bossy and Talay \cite{BossyTalay} for Burgers equation in one-dimensional setting.
On the other hand, general formulations of the drift in \eqref{EDS_MKV} are also considered. In this framework, we mention the work of R\"ockner and Zhang in \cite{RocknerZhang} where the assumption is that the interaction kernel belongs to $L^q([0,T],L^p(R^d))$ for $\frac{d}{p}+\frac{2}{q}<1$. The authors show that there is a unique strong solution to the McKean-Vlasov SDE they consider. Our work falls in the first category. Its main originality with respect to the above works is that the interacting kernel keeps memory of the past, and it is worth of noticing that the drift we consider belongs to $L^q([0,T];L^p(\R^d))$ for $\frac{d}{p}+\frac{2}{q}\geq 1$ and as such does not enter the framework of Röckner and Zhang in \cite{RocknerZhang}. 

From the PDE side, we mention the works of Cong and Liu in \cite{CongLiu} and Lemarié-Rieusset in \cite{lemarie--rieusset2013}. On the one hand, in \cite{CongLiu}, the authors are considering a wider class of Keller-Segel equations, where the first equation in \eqref{EDP_KS} is, in their notations, $\partial_t \rho = \Delta \rho^m - \nabla\cdot(\rho\nabla c)$. Moreover, since they normalized the Keller-Segel equation \eqref{EDP_KS} in such a way that they do not consider the parameter $\chi$, the condition found under which the existence of the solutions to \eqref{EDP_KS} can be proven is on the initial density $u_0$, which is like here supposed to be in $L^{d/2}(\R^d)$. Using a regularization procedure, they prove that the behaviour of $\|\rho_t\|_q$ (for $q>\frac{d}{2}$) is of the form $C\left(t^{-\frac{d}{2\epsilon}(\frac{d}{2}+\epsilon-1)(q-1)}+t^{-\frac{d}{2}(q-1)}\right)^{\frac{1}{q}}$, which seems to be less sharp than the one of the form $C\,t^{-(1-\frac{d}{2q})}$ above. On the other hand, in \cite{lemarie--rieusset2013}, the authors defined a well-suited Morrey space $\dot{\mathcal{M}}^1_{d/2}(\R^d)$ and the initial condition norm $\|u_0\|$ in this space is supposed to be smaller than an unspecified constant $\delta_0$. This way, using the contraction principle on this Morrey space, they are able to show the existence of a weak solution to \eqref{EDP_KS}. It is known that $L^{d/2}(\R^d)\subset \dot{\mathcal{M}}^1_{d/2}(\R^d)$ thus, the space is more general than ours but we are able to explicit the bound associated to the initial condition, and also give the time decay of the norms $\|\rho_t\|_q$ (for $q>\frac{d}{2}$) of the solutions to PDE \eqref{EDP_KS}.

\paragraph{Perspectives}~\\
The probabilistic interpretation of the PDE \eqref{EDP_KS} aims to cast new light upon the problem and propose a new way to analyse the system by analysing the associated stochastic process. Once the probabilistic interpretation is validated, it leads to a microscopic description of the phenomenon in question. Namely, if we consider $N\geq 1$ copies of the process $X$ and plug in the place of the law of the process the empirical measure of the positions of these $N$ particles, we arrive to the following interacting particle system: 
$$dX_t^{i,N}= b_0(t,X_t^{i,N}) dt + \frac{\chi}{N} \sum_{j=1}^N \int_0^t K_{t-s}(X_t^{i,N}-X_s^{j,N}) ds \ dt+ dW_t^i, \quad 1\leq i \leq N, t>0. $$
 This particle system is non-Markovian and the interaction is highly singular. The well-posedness of the system and its convergence towards the limiting process (and the PDE) in one dimensional framework was obtained in Jabir, Talay and Tomasevic \cite{JTT}. However, in $d\geq 2$, these questions are still open due to the singular and non-Markovian interaction the particles undergo.\\
 In addition, the microscopic description leads to a numerical approach to discretize the PDE and the behaviour of particle trajectories could explain how singularities are formed and how they interact. 

\paragraph{Plan of the paper}~\\
First in Section \ref{S_Preliminaries} we introduce the regularization procedure we will use throughout this paper and prove uniform estimates with respect to the regularization parameter for the marginal laws. In Subsections \ref{S_MP} and \ref{S_PDE} we respectively prove Theorem \ref{6.2.3} and Corollary \ref{6.2.5}. In Appendix \ref{S_Annexe} one can find some usefull estimates on the gaussian density norm, and a version of the Gronwall Lemma that will be used throughout the paper. Finally in Appendix \ref{S_properties} we expose further properties of the solutions to \eqref{EDP_KS}.

\section{Density estimates for the regularized process}
\label{S_Preliminaries}
In this section we first present the regularized version of the process \eqref{EDS_MKV}. After ensuring its well-posedness, we present the main result (Theorem \ref{6.3.7}) of this section: Uniform (in the regularization parameter) bounds for the marginal laws of the regularized process. The proof is presented in Subsection~\ref{subsec:proof} and some necessary preparation is done in Subsection~\ref{subsec:prel}.

\paragraph{Regularization procedure} Let $\varepsilon >0$, and for $(t,x)\in \R^+\times \R^d$, we regularize the interaction kernel in the following way: 
$$K_t^\varepsilon(x) :=\dfrac{t^{d/2+1}}{\left(t+\varepsilon\right)^{d/2+1}}K_t(x), \qquad g_t^\varepsilon(x):= \dfrac{t^{d/2}}{\left(t+\varepsilon\right)^{d/2}}g_t(x), \qquad b_0^\varepsilon(t,x):= e^{-\lambda t}(\nabla c_0\ast g_t^\varepsilon)(x).$$
Notice that we only regularize the denominator of $K$ and $g$, and not the exponential part. This particular choice will turn out to be very convenient when the limit when $\varepsilon$ goes to $0$ is taken, see Section \ref{Section3}.\\
The regularized version of \eqref{EDS_MKV} is given by 
\begin{equation}
\label{EDS_epsilon}
\begin{cases}
&dX_t^\varepsilon = dW_t + \chi b_0^\varepsilon(t,X_t^\varepsilon)dt + \chi \left( \int_0^t e^{-\lambda(t-s)} \int K^\varepsilon_{t-s}(X^\varepsilon_t-y)p^\varepsilon_s(y)dy\,ds\right)dt \quad t\leq T,\\
& p^\varepsilon_s(y)dy:=\mathcal{L}(X^\varepsilon_s), s\leq T, \qquad X^\varepsilon_0=\rho_0.\\
\end{cases}
\end{equation}
We  denote its drift as
\begin{equation}
\label{eq:driftNotation}
b^\varepsilon(t,x,p^\varepsilon)=b_0^\varepsilon(t,x)+\chi \int_0^t e^{-\lambda(t-s)} \int K^\varepsilon_{t-s}(x-y)p_s^\varepsilon(y)dy\,ds.
\end{equation}

Notice that there exists $C_\varepsilon>0$ such that for $0<t\leq T$  we have for any $(x,y)\in\R^2$
$$|b_0^\varepsilon(t,x)-b_0^\varepsilon(t,y)|+|K_t^\varepsilon(x)-K_t^\varepsilon(y)|\leq C_\varepsilon|x-y|\qquad \text{and}\qquad |b_0^\varepsilon(t,x)|+|K_t^\varepsilon(x)|\leq C_\varepsilon.$$
One should note that $C_\varepsilon\to\infty$ when $\varepsilon\to 0$. Finally, similar computations as those from Lemma \ref{lemme_normale} show that for all $0<t\leq T$ and $1\leq r \leq\infty$
\begin{align}
\label{6.14}
\|K_t^{\varepsilon,i}\|_r\leq \dfrac{C_1(r)}{(t+\varepsilon)^{\frac{d}{2}(1-\frac{1}{r})+\frac{1}{2}}} \qquad \text{and} \qquad \|g_t^\varepsilon\|_r\leq \dfrac{C_0(r)}{(t+\varepsilon)^{\frac{d}{2}(1-\frac{1}{r})}}.
\end{align}
The above inequalities enable us to apply \cite[Thm. A.1]{tomasevic.2} in combination with the arguments of \cite[Prop. 3.9]{tomasevic.2} (which can easily be adapted to our setting) to prove the following proposition:
\begin{proposition}
\label{EDP_mild_epsilon_thm}
Let $T>0, \lambda >0$, $\chi>0$ and $\varepsilon>0$. Assume that $\rho_0$ is a probability density function that belongs to $L^{\frac{d}{2}}(\R^d)$, and suppose that $ c_0\in W^{1,d}(\R^d)$. Then, (\ref{EDS_epsilon}) admits unique strong solution $(X_t^\varepsilon)_{0\leq t\leq T}$. Besides, the one dimension time marginals of this process have densities $(p^\varepsilon_t)_{0\leq t \leq T}$ with respect to the Lebesgue measure. Finally, for $0<t\leq T$,  $p_t^\varepsilon$ satisfies in the sense of the distribution the following mild equation:
\begin{align}
\label{EDP_mild_epsilon}
p_t^\varepsilon = g_t\ast p_0 - \sum_{i=1}^d\int_0^t \nabla_ig_{t-s}\ast(p_s^\varepsilon b^{\varepsilon,i}(s,\cdot,p^\varepsilon))\ ds.
\end{align}
\end{proposition}
\paragraph{Density estimates} Let us define the norm in which we prove uniform in $\varepsilon>0$ bounds of the family $(p^\varepsilon_s)_{s\leq T}$.
\begin{mydef}
For $\frac{d}{2}\leq r\leq \infty$ and $p=(p_t)_{t\geq 0}$ a family of probability density functions, we define the quantity
\begin{align}
\label{6.17}
\mathcal{N}^r_t(p):=\sup_{0<s<t} s^{1-\frac{d}{2r}} \|p_s\|_r.
\end{align}
with a slight abuse of notations we will also allow us to write, for $\rho=(\rho_t)_{t\geq 0}$ a family of probability laws,
$$\mathcal{N}^r_t(\rho):=\sup_{0<s<t} s^{1-\frac{d}{2r}} \|\rho_s\|_r.$$
\end{mydef}
The main result of this section is the following theorem:
\begin{theorem}
\label{6.3.7}
Let $T>0$, $\varepsilon>0$. Fix a $q\in(d,2d)$ and assume the same hypothesis as in Theorem \ref{6.2.3}.\\
Then there exists $C_q>0$ independent of $\varepsilon>0$ such that $$\mathcal{N}^q_t(p^\varepsilon)\leq C_q.$$
Here the constant $C_q$ is given by 
$$C_q=\dfrac{1-A\chi\|\nabla c_0\|_d - \sqrt{(1-A\chi\|\nabla c_0\|_d)^2-B^2\chi\|p_0\|_{d/2}}}{2K_1 \chi},$$
where $A$ and $B$ are defined in Theorem \ref{6.2.3} and $K_1$ is defined in \eqref{def_K1}.
\end{theorem}

A direct corollary of Theorem \ref{6.3.7} allows us to control the drift of the regularized process independently of $\varepsilon$. Namely, one has 

\begin{cor}
\label{6.3.10}
Under the same assumptions as in Theorem~\ref{6.3.7}, for $q\leq r \leq\infty$ we have 
$$\sup_{t\leq T} t^{\frac{1}{2}-\frac{d}{2r}} \|b^\varepsilon(t,\cdot,p^\varepsilon)\|_r\leq  C_r(\chi,\|\nabla c_0\|_d).$$
\end{cor}
The proof is postponed to Subsection \ref{subsec:prel}

\subsection{Preliminary lemmas}
\label{subsec:prel}
First, we start with the computations that will be frequently used.

\begin{lemma}
\label{lemme_drift}
Let $\lambda\geq 0,  t>0$ and $\frac{d}{2}\leq r\leq\infty$. Then, for $\varepsilon>0$ and $1\leq i \leq d$,  one has that
$$\Big\|\int_0^t e^{-\lambda(t-s)}(K^i_{t-s}\ast p_s^\varepsilon)(\cdot)\ ds\Big \|_r \leq C_1(1) \mathcal{N}_t^r(p^\varepsilon) \beta\left(1-\frac{d}{2r},\frac{1}{2}\right)\frac{1}{t^{\frac{1}{2}-\frac{d}{2r}}}.
$$
\end{lemma}

\begin{proof}
Use the triangular inequality and the convolution inequality. It comes
\begin{align*}
\left\|\int_0^t e^{-\lambda(t-s)}(K^i_{t-s}\ast p^\varepsilon_s)(\cdot)ds\right\|_r \leq \int_0^t \left\| K^i_{t-s}\ast p^\varepsilon_s  \right\|_r ds\leq \int_0^t\| K^i_{t-s}\|_1 \|p^\varepsilon_s\|_r ds.
\end{align*}  
In view of the definition of $\mathcal{N}_t^r(p^\varepsilon)$ and Lemma \ref{lemme_normale} we have
\begin{align*}
\left\|\int_0^t e^{-\lambda(t-s)}(K^i_{t-s}\ast p^\varepsilon)(\cdot)\ ds\right\|_r \leq \mathcal{N}_t^r(p^\varepsilon) \int_0^t\frac{C_1(1)}{\sqrt{t-s}} \frac{1}{s^{1-\frac{d}{2r}}} ds.
\end{align*} 
The desired results follows applying the equality in (\ref{6.3.2}).
\end{proof}

The following result on the linear part of the drift and its regularized version follows from \cite{brezis}[Ex. 4.30], and some simple computations using a convolution inequality.
\begin{lemma}
\label{6.3.3}
Let $t>0$ and $\varepsilon >0$. Then, the function $b^i_0(t,\cdot)$ is continuous on $\R^d$ and for $ d\leq r\leq \infty$ one has
$$\chi\|b^i_0(t,\cdot)\|_r \leq \chi \|\nabla c_0\|_d \frac{C_0\left(\frac{dr}{(d-1)r+d}\right)}{t^{\frac{1}{2}-\frac{d}{2r}}}.$$
Moreover, it holds 
\begin{equation}
\label{6.3.4}
\chi\|b^{\varepsilon,i}_0(t,\cdot)\|_r \leq \chi \|\nabla c_0\|_d \frac{C_0\left(\frac{dr}{(d-1)r+d}\right)}{(t+\varepsilon)^{\frac{1}{2}-\frac{d}{2r}}}.
\end{equation}
\end{lemma}

The next result guarantees that, for a fixed $\varepsilon>0$, the norms we are interested in are well defined. Moreover, it ensures the  continuity of $\mathcal{N}_t^r(p^\varepsilon)$ when $t\to 0$. This result is the adaptation of \cite[Lemma 3.10]{tomasevic.2} to the multidimensional framework we are in. We prefer to postpone its proof to Appendix \ref{proof_6.3.6}.

\begin{lemma}
\label{6.3.6}
Let $0<t\leq T$, $\frac{d}{2}<r<\infty$ and $\varepsilon>0$ fixed. Then it exists $C(\varepsilon,T,\chi)>$ such that 
\begin{align}
\label{6.18}
\mathcal{N}_t^r(p^\varepsilon)\leq C(\varepsilon,T,\chi),
\end{align}
and moreover
\begin{align}
\label{6.19}
\lim_{t\to 0} \mathcal{N}_t^r(p^\varepsilon)=0.
\end{align}
\end{lemma}
We finish this section with the postponed proof of Corollary \ref{6.3.10}.
\begin{proof}[Proof of Corollary \ref{6.3.10} ]
Take $q\in (d, 2d)$ fixed in Theorem \ref{6.3.7}. In view of \eqref{6.3.4} and  Lemma \ref{lemme_drift} for $r=q$, one has 
$$\|b^\varepsilon(t,\cdot,p^\varepsilon)\|_q\leq \dfrac{C_q(\chi,\|\nabla c_0\|_d)}{t^{\frac{1}{2}-\frac{d}{2q}}} + C_q(\chi) \mathcal{N}_t^q(p^\varepsilon) \frac{1}{t^{\frac{1}{2}-\frac{d}{2q}}}.$$
According to Theorem \ref{6.3.7}, $\mathcal{N}_t^q(p^\varepsilon) $ is uniformly bounded w.r.t. $\varepsilon$. Hence, the desired estimate holds for $r=q$.

Now, let $r=\infty$. Apply \eqref{6.3.4} for $r=\infty$ and H\"older's inequality. It comes
$$\|b^\varepsilon(t,\cdot,p^\varepsilon)\|_\infty\leq \dfrac{C_q(\chi,\|\nabla c_0\|_d)}{\sqrt{t}} + C \mathcal{N}_t^q(p^\varepsilon) \int_0^t \|K_{t-s}\|_{q'} \frac{1}{s^{1-\frac{d}{2q}}}  \ ds.$$
In view of Lemma \ref{lemme_normale} and Theorem  \ref{6.3.7}, one has 
$$\|b^\varepsilon(t,\cdot,p^\varepsilon)\|_\infty\leq \dfrac{C_q(\chi,\|\nabla c_0\|_d)}{\sqrt{t}} + C  \int_0^t  \frac{1}{(t-s)^{\frac{d}{2}\left(1-\frac{1}{q'}\right)+\frac{1}{2}}s^{1-\frac{d}{2q}}}  \ ds.$$
The last integral in the above expression is well defined as $q>d$ and hence $\frac{d}{2q}<\frac{1}{2}$. Therefore, in view of (\ref{6.3.2}) one has 
$$\|b^\varepsilon(t,\cdot,p^\varepsilon)\|_\infty\leq \dfrac{C_q(\chi,\|\nabla c_0\|_d)}{\sqrt{t}}.$$
Now, to obtain the desired estimate for any $r\in (q,\infty)$, it suffices to apply the interpolation inequality.
\end{proof}

\subsection{Proof of Theorem~\ref{6.3.7}}
\label{subsec:proof}
Let $q'>1$ such that $\frac{1}{q}+\frac{1}{q'}=1$. Integrate  (\ref{EDP_mild_epsilon}) with respect to a function $f\in L^{q'}(\R^d)$ and apply Hölder inequality. It comes
\begin{align}
\label{6.23}
\left|\int p^\varepsilon_t(x)f(x)dx\right|\leq \|f\|_{q'}\left(\|g_t\ast p_0\|_q + \sum_{i=1}^d \int_0^t \|\nabla_i g_{t-s}\ast(p^\varepsilon_s b^{\varepsilon,i}_s)\|_qds\right).
\end{align}
We denote $A_s^i:=\|\nabla_i g_{t-s}\ast(p^\varepsilon_sb^{\varepsilon,i}_s)\|_q$ and observe that $\frac{1}{q'}+\frac{2}{q}=1+\frac{1}{q}$. Thus, applying a convolution inequality, Lemma \ref{lemme_normale} and  Cauchy-Schwartz inequality, one gets
\begin{equation}
\label{eq:mainth1}
A^i_s\leq \|\nabla_i g_{t-s}\|_{q'}\|p^\varepsilon_sb^{\varepsilon,i}_s\|_{q/2}\leq \dfrac{C_1(q')\|b^{\varepsilon,i}_s\|_qs^{1-\frac{d}{2q}}\|p_s^\varepsilon\|_q}{(t-s)^{\frac{d}{2}(1-\frac{1}{q'})+\frac{1}{2}}s^{1-\frac{d}{2q}}}\leq C_1(q')\mathcal{N}_t^q(p^\varepsilon)\dfrac{\|b^{\varepsilon,i}_s\|_q}{(t-s)^{\frac{d}{2}(1-\frac{1}{q'})+\frac{1}{2}}s^{1-\frac{d}{2q}}}.
\end{equation}
In view of Lemma \ref{6.3.3}
and Lemma \ref{lemme_drift}, it comes
\begin{align*}
\|b^{\varepsilon,i}_s\|_q &\leq 
 \dfrac{C_0\left(\frac{dq}{(d-1)q+d}\right)\chi\|\nabla c_0\|_d +  \chi C_1(1)\mathcal{N}_t^q(p^\varepsilon)\beta\left(1-\frac{d}{2q},\frac{1}{2}\right)}{s^{\frac{1}{2}-\frac{d}{2q}}}.
\end{align*}
Denoting $\tilde{q}_1=\frac{dq}{(d-1)q+d}$ and plugging the latter in \eqref{eq:mainth1}, one has
$$A^i_s \leq \chi C_1(q')\mathcal{N}_t^q(p^\varepsilon)\dfrac{C_0\left(\tilde{q}_1\right)\|\nabla c_0\|_d +  C_1(1)\mathcal{N}_t^q(p^\varepsilon)\beta\left(1-\frac{d}{2q},\frac{1}{2}\right)}{(t-s)^{\frac{d}{2}(1-\frac{1}{q'})+\frac{1}{2}}s^{\frac{3}{2}-\frac{d}{q}}}.$$
Plug the previous inequality in (\ref{6.23}). Now, the fact that $q\in(d,2d)$ ensures that we can apply (\ref{6.3.2}) (as $\frac{3}{2}-\frac{d}{q}<1$ and $\frac{d}{2}(1-\frac{1}{q'})+\frac{1}{2}<1$). Thus, we get
\begin{align*}
\left|\int p^\varepsilon_t(x)f(x)dx\right| & \leq \|f\|_{q'}\Big(\|g_t\ast p_0\|_q \\
& \left. + d\chi C_1(q')\mathcal{N}_t^q(p^\varepsilon)\dfrac{C_0\left(\tilde{q}_1\right)\|\nabla c_0\|_d +  C_1(1)\mathcal{N}_t^q(p^\varepsilon)\beta\left(1-\frac{d}{2q},\frac{1}{2}\right)}{t^{1-\frac{d}{2q}}}\beta\left(\frac{3}{2}-\frac{d}{q},\frac{d}{2}(1-\frac{1}{q'})+\frac{1}{2}\right)\right),
\end{align*}
Notice that  $\|g_t\ast p_0\|_q \leq \|g_t\|_{\tilde{q}_2}\|p_0\|_{\frac{d}{2}}$ with $\tilde{q}_2 = \frac{dq}{d+(d-2)q}$. Then, by Riesz representation theorem, we get
$$\|p^\varepsilon_t\|_q \leq \dfrac{C_0(\tilde{q}_2)}{t^{1-\frac{d}{2q}}}\|p_0\|_{\frac{d}{2}}+d\chi C_1(q') \beta\left(\frac{3}{2}-\frac{d}{q},\frac{d}{2}(1-\frac{1}{q'})+\frac{1}{2}\right)\mathcal{N}_t^q(p^\varepsilon)\dfrac{C_0\left(\tilde{q}_1\right)\|\nabla c_0\|_d+C_1(1)\mathcal{N}_t^q(p^\varepsilon)\beta\left(1-\frac{d}{2q},\frac{1}{2}\right)}{t^{1-\frac{d}{2q}}}.$$
Denote
\begin{align}
\label{def_K1}
& K_1:=dC_1(q')C_1(1)\beta\left(\frac{3}{2}-\frac{d}{q},\frac{d}{2}(1-\frac{1}{q'})+\frac{1}{2}\right)\beta\left(1-\dfrac{d}{2q},\dfrac{1}{2}\right),\\
&\nonumber K_2:=dC_1(q')C_0\left(\tilde{q}_1\right)\beta\left(\frac{3}{2}-\frac{d}{q},\frac{d}{2}(1-\frac{1}{q'})+\frac{1}{2}\right).
\end{align}
By the definition of $\mathcal{N}_t^q(p^\varepsilon)$ we have, after rearranging the terms
\begin{align}
\label{6.24}
0\leq K_1 \chi\, \mathcal{N}_t^q(p^\varepsilon)^2 + (K_2\chi\|\nabla c_0\|_d-1)\mathcal{N}_t^q(p^\varepsilon)+C_0(\tilde{q}_2)\|p_0\|_{\frac{d}{2}}
\end{align}
Now, we define the polynomial
$$P(z):=K_1 \chi\,z^2 + (K_2\chi\|\nabla c_0\|_d-1)z+C_0(\tilde{q}_2)\|p_0\|_{\frac{d}{2}}.$$
Under the condition 
$$\Delta:=(K_2\chi\|\nabla c_0\|_d-1)^2-4K_1C_0(\tilde{q}_2)\|p_0\|_{\frac{d}{2}}\chi>0,$$
$P$ admits two roots, and so under the assumption
$$ K_2\chi\|\nabla c_0\|_d-1<0,$$
the roots are both positive.

Now, remember that from Lemma \ref{6.3.6} we have that $\lim_{t\to 0}\mathcal{N}_t^q(p^\varepsilon)=0$ and from (\ref{6.24}) we have  ${P(\mathcal{N}_t^q(p^\varepsilon))>0}$ for all $t\in[0,T]$. Moreover, in view of (\ref{EDP_mild_epsilon}) the application $t\to \mathcal{N}_t^q(p^\varepsilon)$ is continuous. Then, necessarily $\mathcal{N}_t^q(p^\varepsilon)$ is bounded by the smallest root of $P$ for any $t\in[0,T]$. Since this root does not depend neither of $t$ nor of  $\varepsilon$, this bound is uniform in $\varepsilon$.\\
To conclude, note that the two conditions on the polynomial are equivalents to the condition
$$K_2\chi\|\nabla c_0\|_d + 2\sqrt{K_1 C_0(\tilde{q}_2)\|p_0\|_{\frac{d}{2}}\chi}<1.$$
Denote $A:=K_2$ and $B:=2\sqrt{C_0(\tilde{q}_2)K_1}$ to finish the proof.

\section{Proofs of the main results}
\label{Section3}

\subsection{Martingale problem}
\label{S_MP}
In this subsection we first  obtain the tightness in $\varepsilon\in (0,1)$ of the laws of the solutions to (\ref{EDS_epsilon}). Then, we  prove that the linear and non-linear parts of the drift converge to their respective limits as $\varepsilon \rightarrow 0$. Finally, we give the proof of Theorem~\ref{6.2.3}.
\begin{proposition}
\label{prop_tightness}
Let $T>0$ and $\varepsilon_k=\frac{1}{k}$ for $k\in\N^\ast$. Denote by  $\mathbb{P}^k$ the law of the solution of the SDE~\eqref{EDS_epsilon} associated to $\varepsilon_k$. Then, under the same hypothesis as in Theorem \ref{6.2.3}, the family $(\mathbb{P}^k)_{k\in\N^\ast}$ is tight in $C([0,T],\R^d)$ w.r.t. $k\in\N^\ast$.
\end{proposition}
\begin{proof}
For $m>2$ and $0<s<t\leq T$ we compute,
$$\E|X_t^{\varepsilon_k}-X_s^{\varepsilon_k}|^m \leq C \E\left( \sum_{i=1}^d \left(\int_s^tb^{\varepsilon_k,i}(u,X_u^{\varepsilon_k})du\right)^2\right)^{\frac{m}{2}} +C \E|W_t-W_s|^m.$$
Applying Corollary \ref{6.3.10} with $r=\infty$, we get
$$\E|X_t^{\varepsilon_k}-X_s^{\varepsilon_k}|^m \leq C \left(\int_s^t \dfrac{C_\infty(\chi,\|\nabla c_0\|_2)}{\sqrt{u}}du\right)^m + C(t-s)^{\frac{m}{2}} \leq C (t-s)^{\frac{m}{2}}.$$
Finally, Kolmogorov's criterion (see \cite[Chap.~2, Pb.~4.11]{KaratzasShreve}) implies the desired result. 
\end{proof}

By Proposition \ref{prop_tightness}, one can extract a weakly convergent subsequence from $(\P^k)_{k\geq 1}$. By a slight abuse of notation, we still denote this convergent subsequence by $(\P^k)_{k\geq 1}$ and we denote its limit by $\P^\infty$. Moreover, we recall that by Proposition~\ref{EDP_mild_epsilon_thm} the one dimensional time marginals of $\P^k$ admit densities $p^k_t$. Before passing to the proof of Theorem~\ref{6.2.3}, we prove the following auxilary lemma:
\begin{lemma}
\label{lemme_mainthm}
Let $u>0$. Then, we have $$\lim_{k\rightarrow \infty}\|b_0^{\varepsilon_k,i}(u,\cdot)-b_0^i(u,\cdot)\|_\infty = 0.$$ 
Moreover, supposing that the marginals  $(\P_t^\infty)_{t\geq 0}$ admit probability density functions $(p^\infty_t)_{t\geq 0}$ verifying \eqref{dens_Est_d}
for $q\geq 1$ fixed in Theorem~\ref{6.3.7}, we have that for any $z\in \R^d$
\begin{align}
\label{lemme_mainthm_2}
\lim_{k\to \infty}\left|\int_0^u K^{i,\varepsilon_k}_{u-\tau}\ast p^k_\tau(z)d\tau-\int_0^uK^i_{u-\tau}\ast p^\infty_\tau(z)d\tau\right|=0.
\end{align} 
\end{lemma}

\begin{proof}
For any $x\in\R^d$, applying H\" older's inequality
\begin{align*}
|b_0^{\varepsilon_k,i}(u,x)-b_0^i(u,x)| 
&\leq  C\chi e^{-\lambda u}\left|\int \nabla c_0(x-y)e^{-\frac{y^2}{2u}}dy\right| ~\left( \dfrac{1}{(2\pi u)^{\frac{d}{2}}}-\dfrac{1}{(2\pi (u+\varepsilon_k))^{\frac{d}{2}}} \right) \\
&\leq  \|e^{- \frac{(\cdot)^2}{2u}}\|_{\frac{d}{d-1}} \|\nabla c_0\|_d~\left( \dfrac{1}{u^{\frac{d}{2}}}-\dfrac{1}{(u+\varepsilon_k)^{\frac{d}{2}}} \right).
\end{align*}
Thus, as $u>0$ is fixed, we have  $\lim_{k\rightarrow \infty} \|b_0^{\varepsilon_k,i}(u,\cdot)-b_0^i(u,\cdot)\|_\infty = 0$ .\\
It remains to show (\ref{lemme_mainthm_2}). Fix $z\in \R^d$ and $u>0$ and decompose
\begin{align*}
&~\left|\int_0^u K^{\varepsilon_k,i}_{u-\tau}\ast p^k_\tau(z)d\tau-\int_0^uK^i_{u-\tau}\ast p^\infty_\tau(z)d\tau \right| \\
\leq  & ~\left|\int_0^u K^{\varepsilon_k,i}_{u-\tau}\ast p^k_\tau (z)d\tau- \int_0^u K^i_{u-\tau}\ast p_\tau^k (z)d\tau\right| + \left|\int_0^u K^i_{u-\tau}\ast p^k_\tau (z)d\tau- \int_0^u K^i_{u-\tau}\ast p_\tau^\infty (z)d\tau\right|\\
 =: & ~ A_k+B_k.
\end{align*}
Let's start with $B_k$, for a fixed $\tau\in (0,u)$, the function $K_{u-\tau}$ is continuous and bounded on $\R^d$. Thus, by weak convergence we have that $\lim_{k\rightarrow\infty} (K^i_{u-\tau}\ast p^k_\tau)(z)=(K^i_{u-\tau}\ast p^\infty_\tau)(z)$. In addition, applying Hölder's inequality  for $\frac{1}{q}+\frac{1}{q'}=1$ and  Lemma~\ref{lemme_normale}, we have
$$|K^i_{u-\tau} \ast (p^k_\tau-p^\infty_\tau)|(z) \leq  \|K^i_{u-\tau}\|_{q'} \|p^k_\tau-p^\infty_\tau\|_q \leq  \dfrac{C_q}{(u-\tau)^{\frac{d}{2}(1-\frac{1}{q'})+\frac{1}{2}}\tau^{1-\frac{1}{q}}}.$$
The above bound  is integrable on $(0,u)$ since $q>d$. Thus, by dominated convergence, we conclude that $B_k\rightarrow 0$ as $k\rightarrow \infty$.

We turn to $A_k$.  Apply Hölder's inequality for $\frac{1}{q}+\frac{1}{q'}=1$ and Theorem~\ref{6.3.7}. It comes
$$A_k \leq  \int_0^u \|K^{\varepsilon_k,i}_{u-\tau}-K^i_{u-\tau}\|_{q'} \dfrac{C_q}{\tau^{1-\frac{d}{2q}}}d\tau.$$ First, notice that
$$|K_{u-\tau}^{\varepsilon_k,i}(z) - K_{u-\tau}^i(z)|\leq  \left(\dfrac{1}{(u-\tau)^{\frac{d+1}{2}}}-\dfrac{1}{(u-\tau+\varepsilon_k)^{\frac{d+1}{2}}}\right) |z_i|e^{-\frac{|z|^2}{2(u-\tau)}}.$$ 
Hence,
\begin{align*}
\|K_{u-\tau}^{\varepsilon_k,i} - K_{u-\tau}^i\|_{q'} \leq  C({q'}) \left(\dfrac{1}{(u-\tau)^{\frac{d+1}{2}}}-\dfrac{1}{(u-\tau+\varepsilon_k)^{\frac{d+1}{2}}}\right)(u-\tau)^{\frac{q}{2}}.
\end{align*}
Therefore,  $\lim_{k\to \infty}\|K_{u-\tau}^{\varepsilon_k,i} - K_{u-\tau}^i\|_{q'}=0$ for a fixed $\tau \in (0, u)$. Moreover, in view of Lemma \ref{lemme_normale} and (\ref{6.14}) we have 
\begin{align*}
\|K_{u-\tau}^{\varepsilon_k,i} - K_{u-\tau}^i\|_{q'} \leq  \|K_{u-\tau}^{\varepsilon_k,i}\|_{q'} + \|K_{u-\tau}^i\|_{q'} \leq  \dfrac{C_{q'}}{(u-\tau)^{\frac{d}{2}(1-\frac{1}{q'})+\frac{1}{2}}}.
\end{align*}

Since $q>d$, by dominated convergence, we obtain $A_k\rightarrow 0$ as $k\rightarrow \infty$. \end{proof}

\paragraph{Proof of Theorem \ref{6.2.3}}

We prove here that $\P^\infty$ satisfies \hyperref[defMP]{$\mathcal{(MP)}$}. Notice that \hyperref[defMP]{$\mathcal{(MP)}$}$-i)$  obviously holds as by assumption, for each $k\geq 1$ one has that $\mathbb{P}^k_0(dx)= p_0(x)dx$.   Fix $q$ as in Theorem \ref{6.3.7}, we will now show that  \hyperref[defMP]{$\mathcal{(MP)}$}$-ii)$ and \hyperref[defMP]{$\mathcal{(MP)}$}$-iii)$ hold.
\begin{itemize}
\item[ii)] For all $0<t\leq T$ we define the functionnal $\Lambda_t~:~C_K^\infty(\R^d)\to\R$ by
$$\Lambda_t(\varphi)=\int \varphi(y)\mathbb{P}^\infty_t(dy).$$
Then, using weak convergence of $\P^k$ and the uniform in $k$ estimates on  $\|p_t^k\|_q$ from Theorem~\ref{6.3.7},  $\Lambda_t$ is a bounded operator on $C^\infty_K(\R^d)$. As the latter is a dense subspace of $L^{q'}(\R^d)$, it extends to a bounded operator on $L^{q'}(\R^d)$ and by Riesz's representation theorem, there exists $p_t^\infty \in L^q(\R^d)$  such that $\Lambda_t(\varphi)= \int \varphi(y) p_t^\infty(y)dy$. Hence, $\P^\infty_t(dy)= p^\infty_t(y)dy$. In addition, by weak convergence we have $t^{1-\frac{d}{2q}}\|p_t^\infty\|_q\leq  C_q$, where $C_q$ is defined in Theorem \ref{6.3.7}. 
\item[iii)] Let  $f\in C^\infty_K(\R^d)$. We will show that $(M_t)_{t\geq 0}$, defined in \eqref{def_mart} is a $\mathbb{P}^\infty$-martingale. To do so, we fix $N\geq  1,~0\leq  t_1<\cdots<t_N\leq  s\leq  t\leq  T$ and $\phi\in C_b((\R^d)^N)$, and prove that
\begin{align}
\label{6.26}
\E_{\mathbb{P}^\infty}[(M_t-M_s)\phi(w_{t_1},\cdots,w_{t_N})]=0.
\end{align}
Since $\P^k$ solves the regularized martingale problem with $\varepsilon_k=\frac{1}{k}$ we have that 
\begin{align*}
M_t^k &:= f(w_t)-f(w_0)\\
&-\int_0^t \left[\dfrac{1}{2}\Delta f(w_u)+\nabla f(w_u)\cdot \left( b_0^{\varepsilon_k}(u,w_u)+\chi\int_0^u e^{-\lambda(u-\tau)} (K^{\varepsilon_k}_{u-\tau}\ast p^k_\tau)(w_u)d\tau\right)\right]du,
\end{align*}
is a martingale under $\mathbb{P}^k$ and 
\begin{align*}
0&=\E_{\mathbb{P}^k}[(M_t^k-M^k_s)\phi(w_{t_1},\cdots,w_{t_N})]\\
&=\E_{\mathbb{P}^k}[\phi(\cdots)(f(w_t)-f(w_s))]+\E_{\mathbb{P}^k}[\phi(\cdots)\int_s^t\Delta f(w_u)du]\\
&+\E_{\mathbb{P}^k}[\phi(\cdots)\int_s^t\nabla f(w_u)\cdot b_0^{\varepsilon_k}(u,w_u)du]\\
&+\chi\E_{\mathbb{P}^k}\left[\phi(\cdots)\int_s^t\nabla f(w_u)\cdot \int_0^u e^{-\lambda(u-\tau)} (K^{\varepsilon_k}_{u-\tau}\ast p^k_\tau)(w_u)d\tau\right]du.
\end{align*}
As $(\P^k)_{k\geq 1}$ convergs weakly to $\P^\infty$, the first two terms of the right hand side converge to their analogous terms in (\ref{6.26}). It remains to prove the convergence of the linear and non-linear parts of the drift. We will only show the convergence of the non-linear part. The linear part is done similarly thanks to Lemma~\ref{lemme_mainthm}.

Notice that
\begin{align*}
&\E_{\mathbb{P}^k}\left[\phi(\cdots)\int_s^t \nabla f(w_u)\cdot \int_0^u(K^{\varepsilon_k}_{u-\tau}\ast p^k_\tau)(w_u)d\tau\,du\right]\\
&-\E_{\mathbb{P}^\infty}\left[\phi(\cdots)\int_s^t \nabla f(w_u)\cdot \int_0^u(K_{u-\tau}\ast p^\infty_\tau)(w_u)d\tau\,du\right]\\
&=\left(\E_{\mathbb{P}^k}\left[\phi(\cdots)\int_s^t \nabla f(w_u)\cdot \int_0^u(K^{\varepsilon_k}_{u-\tau}\ast p^k_\tau)(w_u)d\tau\,du\right]\right.\\
&\left.-\E_{\mathbb{P}^k}\left[\phi(\cdots)\int_s^t \nabla f(w_u)\cdot \int_0^u(K_{u-\tau}\ast p^\infty_\tau)(w_u)d\tau\,du\right]\right)\\
&+\left(\E_{\mathbb{P}^k}\left[\phi(\cdots)\int_s^t \nabla f(w_u)\cdot \int_0^u(K_{u-\tau}\ast p^\infty_\tau)(w_u)d\tau\,du\right]\right.\\
&\left.-\E_{\mathbb{P}^\infty}\left[\phi(\cdots)\int_s^t \nabla f(w_u)\cdot \int_0^u(K_{u-\tau}\ast p^\infty_\tau)(w_u)d\tau\,du\right]\right)\\
&:=III_k+IV_k
\end{align*}
We begin with term $IV_k$. To show it converges to zero as $k\to \infty$, we will prove the continuity of the following functional  $H~:~C([0,T],\R^d)\to \R$ defined by:
$$H(x):=\phi(x_{t_1},\cdots,x_{T_N})\int_s^t \nabla f(x_u)\cdot\int_0^u(K_{u-\tau}\ast p^\infty_\tau)(x_u)d\tau\,du.$$
To show the continuity of $H$, we will use the continuity of $K_{u-\tau}$ for $u>\tau$ and apply several times the dominated convergence theorem. Let $x_n \in C([0,T],\R^d)$ such that $x_n\rightarrow x$ in $C([0,T],\R^d)$.
For $\tau<u$, we  have $\lim_{k\to \infty}{K_{u-\tau}^i(x_n(u)-y) = K_{u-\tau}^i(x(u)-y)}$, and  ${|K_{u-\tau}^i(x_n(u)-y)p^\infty_\tau(y)|\leq  \frac{C}{(u-\tau)^{\frac{d+1}{2}}}p^\infty_\tau(y)}$. Dominated convergence implies
 that $$K_{u-\tau}^i\ast p_\tau^\infty(x_n(u)) \longrightarrow K_{u-\tau}^i\ast p_\tau^\infty(x(u)),\quad  n\rightarrow\infty.$$
Applying Hölder inequality, Lemma \ref{lemme_normale} and the estimate from \textit{ii)},  we get
$$|K_{u-\tau}^i\ast p_\tau^\infty(x_n(u))|\leq  \|K_{u-\tau}^i\|_{q'}\|p_\tau^\infty\|_q \leq  \dfrac{C_q}{(u-\tau)^{\frac{d}{2}(1-\frac{1}{q'})+\frac{1}{2}}\tau^{1-\frac{d}{2q}}}.$$
Since $q>d$ the right side is integrable on $[0,u]$, thus we may apply the Lebesgue's theorem a second time. It comes
$$\int_0^u (K_{u-\tau}^i\ast p^\infty_\tau)(x_n(u))d\tau \longrightarrow \int_0^u (K_{u-\tau}^i\ast p^\infty_\tau)(x(u))d\tau , \quad n\rightarrow\infty.$$
Finally, by (\ref{6.3.2}) we have the following inequality:
\begin{align}
\label{proof_mainth_tool}
\left| \nabla f(x_n(u))\cdot\int_0^u (K_{u-\tau}\ast p_\tau^\infty)(x_n(u))d\tau\right|\leq  C_q\|\nabla f\|_\infty\dfrac{\beta(1-\frac{d}{2q},\frac{d}{2}(1-\frac{1}{q'})+\frac{1}{2})}{\sqrt{u}}.
\end{align}
Hence, we may  apply for the last time the dominated convergence and conclude the continuity of $H$. 
\newline
\newline
Now, we turn to $III_k$.  Denoting here $b^i(u,z):=\int_0^uK^i_{u-\tau}\ast p^\infty_\tau(z)d\tau$ and $b^{k,i}(u,z):=\int_0^u K^{\varepsilon_k,i}_{u-\tau}\ast p^k_\tau(z)d\tau$, we have
$$ III_k \leq  \|\phi\|_\infty \int_s^t \sum_{i=1}^d \int |\nabla_if(z)(b^{k,i}(u,z)-b^i(u,z))|p^k_u(z)dz\,du.$$
Applying Hölder inequality, it comes
$$ III_k \leq  \|\phi\|_\infty \int_s^t \dfrac{C_q}{u^{1-\frac{d}{2q}}} \sum_{i=1}^d \left(\int |\nabla_i f(z)|^{q'}|b^{k,i}(u,z)-b^i(u,z)|^{q'}dz\right)^{\frac{1}{q'}}du.$$
For $u>0$ and $i=1,2$, with similar computations to the one done to obtain \eqref{proof_mainth_tool} we have $|b^{k,i}(u,\cdot)|+|b^i(u,\cdot)|\leq  \frac{C}{\sqrt{u}}$. Thus,
$$|\nabla_i f(z)|^{q'}|b^{k,i}(u,z)-b^i(u,z)|^{q'} \leq \dfrac{C}{u^{\frac{q'}{2}}}|\nabla_i f(z)|^{q'}.$$
As $f\in C^\infty_K(\R^d)$ the latter is integrable in $\R^d$. Then, in view of Lemma 4.2, we may apply dominated convergence theorem. It comes,
$$\|\nabla_if(\cdot)(b^{k,i}(u,\cdot)-b^i(u,\cdot))\|_{q'}\rightarrow 0, \quad k\rightarrow\infty.$$
Finally, applying Hölder inequality, we have, by Lemma \ref{lemme_drift} and part ii),
$$\|\nabla_i f(\cdot)(b^{k,i}(u,\cdot)-b^i(u,\cdot))\|_{q'} \leq \|\nabla_i f(\cdot)\|_r \|b^{k,i}(u,\cdot)-b^i(u,\cdot)\|_q  \leq  C_K \|\nabla f\|_\infty \dfrac{C}{u^{\frac{1}{2}-\frac{d}{2q}}}.$$
Hence, 
$$\dfrac{1}{u^{1-\frac{d}{2q}}} \|\nabla_i f(\cdot)(b^{k,i}(u,\cdot)-b^i(u,\cdot))\|_{q'}\leq  \|\nabla f\|_\infty\dfrac{C}{u^{\frac{1}{2}-\frac{d}{2q}+1-\frac{d}{2q}}}.$$
Since $q<2d$, by dominated convergence we have $III_k\rightarrow 0$ for $k\rightarrow\infty$, thus $M_t$ is, as expected, a $\mathbb{P}^\infty$-martingale. 
\end{itemize}

\subsection{Solution to the PDE \eqref{EDP_KS}}
\label{S_PDE}

In this Subsection we will prove Corollary \ref{6.2.5} and Proposition \ref{6.2.5.1}. We will prove it for $\lambda=0$ as the case where $\lambda>0$ follows the same arguments. 

\paragraph{Proof of Corollary \ref{6.2.5}}

We will begin by proving existence of solutions to (\ref{EDP_KS}).\\
We will note by $\rho(t,\cdot)=p_t(\cdot)$ the time marginals of the law of the solution to (\ref{EDS_MKV}) constructed in the existence part of the proof of Theorem \ref{6.2.3}.\\
We have immediatly that $\rho$ verifies, for a $d<q<2d$,
$$\sup_{t\leqslant T} t^{1-\frac{d}{2q}}\|\rho(t,\cdot)\|_q \leqslant C_q.$$
Then define $c(t,x)$ as in Definition \ref{6.2.4}, which is well defined in view of Lemma \ref{lemme_normale}, the density estimates and identity (\ref{6.3.2}), indeed
$$|c(t,x)|\leqslant \dfrac{\|c_0\|_d}{\sqrt{t}}+C\int_0^t \|\rho(t-s,\cdot)\|_q\|g_s\|_{\frac{q}{q-1}}ds\leqslant\dfrac{\|c_0\|_d}{\sqrt{t}}+ C\beta\left(1-\dfrac{d}{2q},\dfrac{d}{2q}\right).$$
Then, remembering that $c_0\in L^d(\R^d)$ and the density estimates for $\rho_s$, one can easily show that $c(t,\cdot)$ is in $L^2(\R^d)$ too. Moreover as $g_t$ is differentiable when $t>0$ and remembering the density estimates on $\rho(t,\cdot)$, $\nabla c(t,x)$ exists for $x\in\R^d$, and we have
$$\nabla_i c(t,x) = \nabla_i(g(t,\cdot)\ast c_0)(x) + \int_0^t (\rho_{t-s}\ast\nabla_i g(s,\cdot))(x)ds.$$
Finally note that, using the same arguments than in the proof of Lemma \ref{6.3.10}, one can show that the drift $b$ defined in \eqref{def_b} verifies for all $q\leqslant r\leqslant\infty$ 
$$t^{\frac{1}{2}-\frac{d}{2r}}\|b(t,\cdot,\rho)\|_r\leqslant C_r.$$
This control on the drift allows us to follow the same path as in the proof of (\cite{tomasevic.2}, Proposition 3.9 page 13) and come to the mild equation, for all $f\in C^\infty_K(\R^d)$ and all $t\in(0,T]$
\begin{align*}
\int f(y)\rho(t,y)dy &= \int f(y)(g_t\ast\rho_0)(y)dy\\
&-\sum_{i=1}^d \int f(y) \int_0^t [\nabla_i g_{t-s}\ast (b^i(s,\cdot,\rho)\rho(s,\cdot))](y)ds\,dy.
\end{align*}
Thus, $\rho$ verifies in the sense of the distributions
\begin{align}
\label{6.32}
\rho(t,\cdot)= g_t\ast \rho_0 - \sum_{i=1}^d\int_0^t \nabla_i g_{t-s}\ast (b^i(s,\cdot,\rho)\rho(s,\cdot))ds.
\end{align}
One should still verify that $\nabla c_0$ exists and that $b=\nabla c_0$ to complete the proof.\\
Yet as $c_0\in W^{1,d}(\R^d)$, we have $\nabla_i (g_t\ast c_0)=g_t\ast\nabla_i c_0$. Finally $\chi \nabla_i c(t,x)$ is exactly the drift of (\ref{6.32}), and the pair $(\rho,c)$ satisfies indeed Definition \ref{6.2.4}.

\paragraph{Proof of Proposition \ref{6.2.5.1}}
We will only give a sketch of the proof, since the following stays close to what can be found in \cite{tomasevic.2}.\\
Let us take $(\rho^1,c^1)$ and $(\rho^2,c^2)$ be two solutions of (\ref{EDP_KS}) and our goal is now to show that $\rho^1=\rho^2$. It is enough because from the expression for $\nabla c^i(t,\cdot)$ one can deduce that $c^1=c^2$.\\
Because of Definition \ref{6.2.4} we have that 
$$\exists \, d < q < d, ~~\exists C_q ~~:~~\sup_{t\leqslant T} t^{1-\frac{d}{2q}}\|\rho^i(t,\cdot)\|_q \leqslant C_q.$$
And we would like to prove that $\sup_{t\leqslant T} \|\rho^1(t,\cdot)-\rho^2(t,\cdot)\|_q=0$. To do this, we will note ${f(T):=\sup_{t\leqslant T} \|\rho^1(t,\cdot)-\rho^2(t,\cdot)\|_q}$.\\
One should note that
\begin{align*}
\|\rho^1(t,\cdot)-\rho^2(t,\cdot)\|_q &\leqslant \chi \sum_{i=1}^d \int_0^t \|\nabla_i g_{t-s} \ast [\nabla_i c^1(s,\cdot)(\rho^1(s,\cdot)-\rho^2(s,\cdot))]\|_q ds\\
&+ \chi \sum_{i=1}^d \int_0^t \|\nabla_i g_{t-s} \ast [\rho^2(s,\cdot)(\nabla_i c^1(s,\cdot)-\nabla_i c^2(s,\cdot))]\|_q ds := I+II.
\end{align*}
Since the key is to control the behaviour of $\nabla c^i_t$, applying two times Hölder's inequality, one can show that 
\begin{align*}
\|\nabla c_t^i\|_\infty \leqslant \dfrac{\|\nabla c_0\|_d}{\sqrt{t}}+C_1(q')\beta\left(\frac{d}{2q}+\frac{1}{2}, 1 - \frac{d}{2q}\right) \dfrac{C_q}{\sqrt{t}},
\end{align*}
where $\beta(\cdot,\cdot)$ is defined in \eqref{6.3.2} and $C_1(\cdot)$ in Lemma \ref{lemme_normale}. Moreover, one can also show that, using a convolution inequality,
\begin{align*}
\|\nabla_i c^1(s,\cdot)-\nabla_ic^2(s,\cdot)\|_{\infty} \leqslant \frac{C_1(q')}{\frac{3}{2}+\frac{d}{2q}} t^{\frac{1}{2}-\frac{d}{2q}} f(T),
\end{align*}
because $q'<\frac{d}{d-1}$ since $q>d$.\\\\
Following the steps described in \cite{tomasevic.2} and keeping in mind our precedent control on $\|\nabla c^i(t,\cdot)\|_\infty$, one can show that
\begin{align*}
I \leqslant \chi C_1(1) \beta\left(\frac{1}{2},\frac{1}{2}\right) \left(\|\nabla c_0\|_d + C_q C_1(q')\beta\left(\frac{d}{2q}+\frac{1}{2}, 1 - \frac{d}{2q}\right)\right) f(T).
\end{align*}
And, with our control on $\|\nabla_i c^1(s,\cdot)-\nabla_ic^2(s,\cdot)\|_{\infty}$, one can show that
\begin{align*}
II\leqslant \chi C_1(1)C_1(q') C_q \beta\left(\frac{1}{2},\frac{1}{2}\right) f(T).
\end{align*}
Finally, using our two estimates and taking the $\sup$ over $\{t\leqslant T\}$, we come to
$$f(T) \leqslant \chi  C(\|\nabla c_0\|_d + C_q) f(T) + \chi C C_q f(T) \leqslant \chi  C(\|\nabla c_0\|_d + C_q) f(T).$$
Thus if $\chi  C(\|\nabla c_0\|_d , C_q) f(T)<1$, where 
$$C(\|\nabla c_0\|_d , C_q) := C_1(1) \beta\left(\frac{1}{2},\frac{1}{2}\right) \left[ \|\nabla c_0\|_d + C_q C_1(q')\left(\beta\left(\frac{d}{2q}+\frac{1}{2}, 1 - \frac{d}{2q}\right) + 1 \right) \right] ,$$
one has the desired result. Remembering that $\chi C_q \rightarrow 0$ for $\chi \rightarrow 0$, we have the condition \eqref{condition_2}.

\appendix

\section{Some technicalities}
\label{S_Annexe}

\subsection{Properties of the Gaussian distribution}
Simple computations lead to the following lemma :
\begin{lemma}
\label{lemme_normale}
For $t>0$ we consider the gaussian density $g_t$ and its gradient $\nabla g_t$ defined in \eqref{normale}.\\
Then, for $1\leq r \leq \infty$, one has
$$\|g_t\|_r = \dfrac{C_0(r)}{t^{\frac{d}{2}(1-\frac{1}{r})}},$$
and, for all $i = 1,\dots,d$,
$$\|\nabla_i g_t\|_r = \dfrac{C_1(r)}{t^{\frac{d}{2}(1-\frac{1}{r})+\frac{1}{2}}},$$
where $C_0(r):=(2\pi )^{-\frac{d}{2}(1-\frac{1}{r})}$ and $C_1(r) = \Gamma\left(\frac{r+1}{2}\right)^{\frac{1}{r}} \left[2^{\frac{d}{2}(1-\frac{1}{r})+\frac{1}{2}}\pi^{\frac{d}{2}(1-\frac{1}{r})+\frac{1}{2r}}r^{\frac{1}{2}+\frac{1}{2r}}\right]^{-1}$.
\end{lemma}

\subsection{A special Gronwall Lemma}

We will here show a version of the Grönwall Lemma adapted to our setting.
\begin{lemma}
\label{lemme_Tomasevic}
Let $0\leq  \alpha <1$ and $(u(t))_{t\geq  0}$ a bounded positive function such that for a $T>0$ there exists $C_T>0$ such that $0<t\leq  T$
$$u(t)\leq  C_T + C_T t^\alpha \int_0^t\dfrac{u(s)}{\sqrt{t-s}s^\alpha}ds.$$
Then there exists $C>0$ such that $u(t)\leq 	C$ for all $0<t\leq  T$.
\end{lemma}
\begin{proof}
Iterating the relation we get
$$u(t)\leq  C_T + C_T t^\alpha \int_0^t\dfrac{1}{s^\alpha\sqrt{t-s}}ds + C_T^2 t^\alpha \int_0^t\int_0^s\dfrac{u(r)}{\sqrt{s-r}r^\alpha \sqrt{t-s}}dr ~ds.$$
By (\ref{6.3.2}) and with Fubini's theorem we have
$$u(t)\leq  C_T +C_T t^\alpha \int_0^t \dfrac{u(r)}{r^\alpha}\int_r^t\dfrac{1}{\sqrt{(t-s)(s-r)}}ds ~dr.$$
Using (\ref{6.3.2}) again we get finally to
$$u(t)\leq C_T + C_T \int_0^t\dfrac{u(r)}{r^\alpha}dr.$$
We conclude by using classical Grönwall Lemma.
\end{proof}
\subsection{Proof of Lemma \ref{6.3.6}}
\label{proof_6.3.6}
We first need the following result: 
\begin{lemma}
\label{6.3.5}
Let $p_0$ a probability density function such that $p_0\in L^{\frac{d}{2}}$ and let $\frac{d}{2}<r<\infty$. Then, we have
$$\limsup_{t\to 0} t^{1-\frac{d}{2r}}\|g_t\ast p_0\|_r=0.$$
\end{lemma}

\begin{proof}
Applying a convolution inequality, one has for all $f\in C_K(\R^d)$,
\begin{align*}
t^{1-\frac{d}{2r}}\|g_t\ast p_0\|_r &\leq t^{1-\frac{d}{2r}} \|g_t\ast f\|_r+t^{1-\frac{d}{2r}}\|g_t\ast(p_0-f)\|_r\\
&\leq t^{1-\frac{d}{2r}} \|g_t\|_1\|f\|_r + t^{1-\frac{d}{2r}} \|g_t\|_{\frac{dr}{d+r(d-2)}} \|p_0-f\|_{\frac{d}{2}}\\
&\leq \|f\|_rt^{1-\frac{d}{2r}}+C\|p_0-f\|_{\frac{d}{2}}.
\end{align*}
Thus, we get
$$\limsup_{t\to 0} t^{1-\frac{d}{2r}} \|g_t\ast p_0\|_r \leq C\|p_0-f\|_{\frac{d}{2}}.$$
As $C_K(\R^d)$ is dense in $L^{\frac{d}{2}}(\R^d)$, the right hand side can be made arbitrary small.
\end{proof}

Now, we can prove Lemma \ref{6.3.6}.
\begin{proof}[Proof of Lemma \ref{6.3.6}]
Notice that the drift of the regularized SDE \eqref{EDS_epsilon} is bounded
\begin{align}
\label{6.3.6.bound}
\|b^{\varepsilon,i}(t,\cdot,p^\varepsilon)\|_\infty \leq \dfrac{C}{\sqrt{\varepsilon}}+\dfrac{tC}{\varepsilon^{\frac{d+1}{2}}}=C(\varepsilon,T,\chi).
\end{align}
Integrating the mild equation (\ref{EDP_mild_epsilon}) with respect to a function $f$ in $L^{r'}(\R^d)$, Hölder's inequality gives
$$
\left|\int p^\varepsilon_t(x)f(x)dx\right|\leq \|f\|_{r'}\left(\|g_t\ast p_0\|_r + \sum_{i=1}^d \int_0^t \|\nabla_i g_{t-s}\ast(p^\varepsilon_s b^{\varepsilon,i}_s)\|_rds\right).$$
We have, using the Riesz representation theorem,
\begin{align}
\label{6.20}\|p^\varepsilon_t\|_r \leq \|g_t\ast p_0\|_r + \sum_{i=1}^d \int_0^t \|\nabla_i g_{t-s}\ast(p^\varepsilon_s b^{\varepsilon,i}_s)\|_rds.
\end{align}
Using a convolution inequality and \eqref{6.3.6.bound}, and then with Lemma \ref{lemme_normale} we get
\begin{align*}
\|p^\varepsilon_t\|_r & \leq \|g_t\ast p_0\|_r + C(\varepsilon,T,\chi)\sum_{i=1}^d \int_0^t \|\nabla_i g_{t-s}\|_{1} \|p^\varepsilon_s\|_r ds \\
& \leq \|g_t\ast p_0\|_r + C(\varepsilon,T,\chi) \int_0^t \dfrac{C_1(1)}{\sqrt{t-s}} \|p_s^\varepsilon\|_r ds .
\end{align*}
Thus, we have
\begin{align}
\label{6.21.1}
t^{1-\frac{d}{2r}}\|p^\varepsilon_t\|_r \leq t^{1-\frac{d}{2r}}\|g_t\ast p_0\|_r + C(\varepsilon,T,\chi) t^{1-\frac{d}{2r}} \int_0^t \dfrac{1}{\sqrt{t-s}}\dfrac{\mathcal{N}_s^r(p^\varepsilon)}{s^{1-\frac{d}{2r}}}ds
\end{align}
To obtain (\ref{6.18}) one should use a convolution inequality and Lemma \ref{lemme_normale}, while using that $p_0 \in L^{d/2}(\R^d)$, and finally use Lemma \ref{lemme_Tomasevic} with $u(t):=t^{1-\frac{d}{2r}}\|p_t^\varepsilon\|_r$. Then, since the constant does not depend on $t$, we obtain the desired result.\\
To obtain (\ref{6.19}) one should apply Lemma \ref{6.3.5} and the fact that an integral is continuous.
\end{proof}

\newpage

\section{Further properties of the solution of \eqref{EDP_KS}}
\label{S_properties}

In this section we will prove some properties of $\rho=(\rho_t)_{t\geq 0}$ the solution to PDE (\eqref{EDP_KS}) we have just constructed in Theorem \ref{6.2.5}, where, with a slight abuse of notations, we have noted $\rho_t:=\rho(t,\cdot)$. The main theorem is the following
\begin{theorem}
\label{S_properties_mainthm}
Let $T>0$.\\
Under the same assumptions as in Theorem \ref{6.2.5} and with $(\rho_t,c_t)_{t\geq 0}$ a solution to \eqref{EDP_KS} in the sense of Definition \ref{6.2.4}, we have, for $r \in (1,\frac{d}{2}]$, 
$$  \|\rho_t\|_r \leqslant C_r,$$
and for $r\in (\frac{d}{2},\infty)$,
$$ \mathcal{N}^r_t(\rho)\leq C_r.$$
\end{theorem}
In a first time we only prove those properties locally, ie for $T\leq 1$, in Lemma \ref{estimation_densite_borne_lemme}, Lemma \ref{6.3.9} and Lemma \ref{6.3.9.1}, and finally we will show how to generalize to any $T>0$ in subsection \ref{Ss_properties}.

\subsection{Local properties}
\begin{lemma}
\label{estimation_densite_borne_lemme}
Let $T\leq 1$.\\
Under the same assumptions as in Theorem \ref{6.2.5} and with $(\rho_t,c_t)_{t\geq 0}$ a solution to \eqref{EDP_KS} in the sense of Definition \ref{6.2.4}, we have, for $r \in (1,\frac{d}{2}]$, $$  \|\rho_t\|_r \leqslant C_r,$$ where $C_r$ does not depend on $T\leq 1$.
\end{lemma}

\begin{proof}
To simplify the computations, we will fix $q$ in Theorem \ref{6.3.7} such that $q=\frac{3d}{2}$. 
\subparagraph{Step 1}
Note that in view of Interpolation inequality from \cite[p.93]{brezis}, we have already a control on $\sup_{t\geq 0} \|\rho_t\|_r$, for $r\in[1,q]$. Indeed, we know that $\|\rho_t\|_1=1$ for all $t\geq 0$, and by the Definition \ref{6.2.4}, we have $\|\rho_t\|_q \leqslant \frac{C_q}{t^{1-\frac{d}{2q}}}=\frac{C_q}{t^{\frac{2}{3}}}$ for all $t\geq 0$. Thus, we get for $1\leqslant r\leqslant \frac{3d}{2}$,
\begin{align}
\label{estimation_densite_borne_r}
\|\rho_t\|_r \leqslant \|\rho_t\|_q^\theta \leqslant \dfrac{C_q^\theta}{t^{\frac{2}{3}\theta}}, \qquad \forall t\geq 0
\end{align}
where $\theta=\dfrac{1-\frac{1}{r}}{1-\frac{1}{q}}$. In particular, taking $r=\frac{d}{2}$ and $r=d$ in the above, we have that there exists $ A_0>0$ such that
\begin{align}
\label{estimation_densite_borne_2d_d}
\|\rho_t\|_{\frac{d}{2}}\leqslant \frac{A_0}{t^{a_0'}} \qquad \text{and} \qquad \|\rho_t\|_d\leqslant \frac{A_0}{t^{a_0}} \qquad \forall t\geq 0
\end{align}
with $a_0'=\frac{2}{3}(\frac{1-\frac{2}{d}}{1-\frac{2}{3d}})$ and $a_0=\frac{2}{3}(\frac{1-\frac{1}{d}}{1-\frac{2}{3d}})$.
\subparagraph{Step 2}
In \textbf{Step 3} we will use the relationship in (\ref{6.10}) in order to improve the bounds in \eqref{estimation_densite_borne_2d_d}.\\
Assume for a moment that we have proved by induction that, for all $n\geq 1$, there exist $A_n'>0$ and $A_n>0$ such that
\begin{align}
\label{estimation_densite_borne_2d_d_step2.k}
\|\rho_t\|_{\frac{d}{2}}\leqslant \frac{A_n'}{t^{a_n'}} \qquad \text{and} \qquad \|\rho_t\|_d\leqslant \frac{A_n}{t^{a_n}} \qquad \forall t\geq 0,
\end{align}
where $a_n':=\frac{2a_0-1}{2^{n-1}}$ and $a_n:=\frac{2a_0-1}{2^{n+1}}+\frac{1}{2}$. Moreover,assume as well that $(A_n)_{n\geq 1}$ and $(A_n')_{n\geq 1}$ satisfy
\begin{align}
\label{estimation_densite_borne_2d_d_step2.k_relation}
A_{n+1}' := &\|\rho_0\|_{d/2} + d \chi\|\nabla c_0\|_d \beta_{\frac{1}{6}}A_{n} + d\chi \beta_{\frac{1}{6}}^2 A_{n}^2,\\
\label{estimation_densite_borne_2d_d_step2.k_relation_2}
A_{n+1} := & A_{n+1}'^{\frac{1}{4}}~C_q^{\frac{3}{4}},
\end{align}
where $\beta_{\frac{1}{6}}$ is defined in \eqref{6.3.2.1}.\\
Then, in \eqref{estimation_densite_borne_2d_d_step2.k} take the limit as $n \rightarrow \infty$. Since $a_n'\rightarrow 0$ as $n\rightarrow\infty$, it only remains to prove that $\sup_{n\geq 1} A_n'<\infty$. Then, we will get that there exists $C>0$ such that $\sup_{t\geq 0}\|\rho_t\|_{\frac{d}{2}}\leq C$. Our desired conclusion then follows from Interpolation inequality from \cite[p.93]{brezis}.\\
To do so, we study the sequence defined by iteration $A_{n+1}=f(A_n)$ where 
$$f~:~x~\mapsto \left(\|\rho_0\|_{d/2} + d \chi\|\nabla c_0\|_d \beta_{\frac{1}{6}}x + d\chi \beta_{\frac{1}{6}}^2 x^2\right)^{\frac{1}{4}}~C_q^{\frac{3}{4}}.$$
We know that $f$ is positive, continuous, strictly increasing, verifies $f(0)>0$ and grows as a $O(\sqrt{x})$ for large values of $x$. Thus, $f$ has at least one fixed point. Moreover, solving $f(x)=x$ is equivalent to finding the roots of a polynomial of degree 4, thus, we have at most 5 fixed points. Let us denote $y$ the biggest one, $I:=[0,y]$ and $J:=[y,\infty)$. Since we said that $f$ is increasing, we have that $f(I)\subset I$ and $f(J)\subset J$. Thus, if $A_1 \in I$, the sequence $(A_n)_{n\geq 1}$ is in $I$ and thus bounded above by $y$. And if $A_1 \in J$, the sequence $(A_n)_{n\geq 1}$ is in $J$, and since $f(x) < x$ on $J$ (because of the fact that on $J$, $f$ grows as a $O(\sqrt{x})$ for large values of $x$), the sequence $(A_n)_{n\geq 1}$ is decreasing, and bounded above by $A_1$. Thus, in each case $(A_n)_{n\geq 1}$ is bounded above, and since $A_{n} = A_{n}'^{\frac{1}{4}}~C_q^{\frac{3}{4}}$, $(A_n')_{n\geq 1}$ is bounded above by a $C>0$ too.

\subparagraph{Step 3} It remains to prove the recurrent relationship given in Step 2.\\
First, we show that \eqref{estimation_densite_borne_2d_d_step2.k}, \eqref{estimation_densite_borne_2d_d_step2.k_relation} and \eqref{estimation_densite_borne_2d_d_step2.k_relation_2} hold for $n=1$, starting from \eqref{estimation_densite_borne_2d_d}. From (\ref{6.10}), we have, by convolution inequality and Cauchy-Schwarz inequality
\begin{align}
\label{estimation_densite_borne}
\nonumber
\|\rho_t\|_{d/2} & \leqslant \|g_t\ast \rho_0\|_{d/2} +  \sum_{i=1}^d \int_0^t \|\nabla_i g_{t-s}\ast(\rho_s b^{i}_s)\|_{d/2} ds\\
& \leqslant  \|\rho_0\|_{d/2} + d \int_0^t \|\nabla_1 g_{t-s} \|_{1} \|\rho_s\|_{d} \|b^{1}_s\|_{d} ds.
\end{align}
Combining a convolution inequality, Lemma \ref{lemme_normale}, (\ref{estimation_densite_borne_2d_d}), and equality \eqref{6.3.2}, we have
\begin{align}
\nonumber
\|b^{1}_s\|_d & \leqslant \chi \|g_s \ast \nabla c_0\|_d + \chi  \int_0^s \|K_{s-r}\|_1 \|p_r\|_d dr \\
\nonumber
&\leqslant \chi\|\nabla c_0\|_d + \chi C_1(1)A_0 \int_0^s \dfrac{1}{\sqrt{s-r}}\dfrac{1}{r^{a_0}}dr  \\
\label{estimate_b}
&= \chi\|\nabla c_0\|_d + \chi C_1(1)A_0\beta\left(a_0,\frac{1}{2}\right) \dfrac{1}{s^{a_0-\frac{1}{2}}}.
\end{align}
Then, plug \eqref{estimate_b} in (\ref{estimation_densite_borne}) and use the estimate in (\ref{estimation_densite_borne_2d_d}) and equality \eqref{6.3.2}. It comes 
\begin{align}
\label{estimation_densite_borne_2}
\|\rho_t\|_{d/2} & \leqslant  \|\rho_0\|_{d/2} + d \int_0^t \frac{C_1(1)}{\sqrt{t-s}} \frac{A_0}{s^{a_0}} \left( \chi\|\nabla c_0\|_d + \chi C_1(1)A_0\beta\left(a_0,\frac{1}{2}\right) \dfrac{1}{s^{a_0-\frac{1}{2}}} \right) ds\\
\nonumber
&\leqslant \|\rho_0\|_{d/2} + \left[dC_1(1)A_0\chi\|\nabla c_0\|_d \beta\left(a_0,\frac{1}{2}\right)\right]\frac{1}{t^{a_0-\frac{1}{2}}} + \left[d\chi C_1(1)^2A_0^2\beta\left(a_0,\frac{1}{2}\right)\beta\left(2a_0-\frac{1}{2},\frac{1}{2}\right)\right] \dfrac{1}{t^{2a_0-1}}.
\end{align}
Since $T\leq 1$, we have that $\frac{1}{t^{a_0-\frac{1}{2}}}\leq \frac{1}{t^{2a_0-1}}$, and therefore
$$\|\rho_t \|_{\frac{d}{2}} \leqslant \dfrac{A_1'}{t^{2a_0-1}}, $$
where $$A_1'=\|\rho_0\|_{d/2} + d A_0\chi\|\nabla c_0\|_d \beta_{\frac{1}{6}} + d\chi A_0^2 \beta_{\frac{1}{6}}^2.$$
Thus, interpolating between $\|\rho_t\|_{\frac{d}{2}}$ and $\|\rho_t\|_q$ we obtain
$$\|\rho_t\|_d \leqslant \|\rho_t\|_{\frac{d}{2}}^{\frac{1}{4}} ~ \|\rho_t\|_q^{\frac{3}{4}} \leqslant \frac{A_1'^{\frac{1}{4}}}{t^{\frac{1}{4}(2a_0-1)}} ~ \frac{C_q^{\frac{3}{4}}}{t^{\frac{3}{4}\frac{2}{3}}}=\frac{A_1}{t^{\frac{a_0}{2}+\frac{1}{4}}},$$
where $A_1 := A_1'^{\frac{1}{4}}~C_q^{\frac{3}{4}}.$
\newline
Hence, \eqref{estimation_densite_borne_2d_d_step2.k}, \eqref{estimation_densite_borne_2d_d_step2.k_relation} and \eqref{estimation_densite_borne_2d_d_step2.k_relation_2} hold for $n=1$.\\
Finally, suppose that \eqref{estimation_densite_borne_2d_d_step2.k}, \eqref{estimation_densite_borne_2d_d_step2.k_relation} and \eqref{estimation_densite_borne_2d_d_step2.k_relation_2} hold for a fixed $n\geq 1$. Repeat the above calculations using \eqref{estimation_densite_borne_2d_d_step2.k} instead of \eqref{estimation_densite_borne_2d_d}. We obtain that the desired relations hold for $n+1$.
\end{proof}

We can now control the quantity $\mathcal{N}^r_t(\rho)$ defined in Definition \ref{6.17} for the different $\frac{d}{2}\leq r \leq \infty$.

\begin{lemma}
\label{6.3.9}
Let $T\leq 1$.\\
Under the same assumptions as in Theorem \ref{6.2.5} and with $(\rho_t,c_t)_{t\geq 0}$ a solution to \eqref{EDP_KS} in the sense of Definition \ref{6.2.4}, we have for $\frac{d}{2} \leqslant r<q$,
$$\mathcal{N}^r_t(\rho)\leqslant C_r.$$
\end{lemma}

\begin{proof}
We note that, according to Lemma \ref{estimation_densite_borne_lemme}, $\|\rho_t\|_{d/2} \leqslant C$ and with Lemma \ref{6.3.7} we know that $\|\rho_t\|_q \leqslant \dfrac{C_q}{t^{1-\frac{d}{2q}}}$. Then, with interpolation's inequality from \cite[p.93]{brezis} we have, with $\frac{1}{r}=\frac{2}{d}\alpha + \frac{1-\alpha}{q}$, and thus $1-\alpha = \dfrac{1-\frac{d}{2r}}{1-\frac{d}{2q}}$,
$$\|\rho_t\|_r \leqslant C^\alpha \dfrac{C_q^{1-\alpha}}{t^{(1-\alpha)\left(1-\frac{d}{2q}\right)}}=\dfrac{C_r}{t^{1-\frac{d}{2r}}}.$$
\end{proof}

\begin{cor}
\label{6.3.10-2}
Let $T\leq 1$.\\
Under the same assumptions as in Theorem \ref{6.2.5} and with $(\rho_t,c_t)_{t\geq 0}$ a solution to \eqref{EDP_KS} in the sense of Definition \ref{6.2.4}, for $d\leqslant r \leqslant\infty$ we have
$$\|b_t\|_r\leqslant \dfrac{C_r(\chi,\|\nabla c_0\|_d)}{t^{\frac{1}{2}-\frac{d}{2r}}},$$
where $b$ is defined by \eqref{def_b}.
\end{cor}

\begin{proof}
According to Lemma \ref{6.3.3} we have
\begin{align}
\label{6.25}
\|b^{i}_t\|_r \leqslant \dfrac{C(\chi,\|\nabla c_0\|_d)}{t^{\frac{1}{2}-\frac{d}{2r}}}+ \chi \int_0^t \|K_{t-s}^{i}\ast \rho_s\|_rds.
\end{align}
Now we have two cases
\begin{itemize}
\item[$\bullet$] If $d\leqslant r<q$, then Corollary \ref{6.3.9} implies immediately
$$\|K_{t-s}^{i}\ast \rho_s\|_r \leqslant \dfrac{C_r}{\sqrt{t-s}s^{1-\frac{d}{2r}}}.$$
\item[$\bullet$] If $q\leqslant r\leqslant \infty$, then we take $p$ such that $\frac{1}{p}=1+\frac{1}{r}-\frac{1}{q}$, and we note that $d<q\leqslant r$ and $\frac{2}{d}<\frac{1}{p}\leqslant 1$. Applying interpolation's inequality from \cite[p.93]{brezis} and Corollary \ref{6.3.9} we get 
$$\|K_{t-s}^{i}\ast \rho_s\|_r \leqslant \dfrac{C_r}{(t-s)^{2-\frac{d}{2p}}s^{1-\frac{d}{2q}}}.$$
\end{itemize}
To conclude, in each case we use \eqref{6.3.2}.
\end{proof}

\begin{lemma}
\label{6.3.9.1}
Let $T\leq 1$.\\
Under the same assumptions as in Theorem \ref{6.2.5} and with $(\rho_t,c_t)_{t\geq 0}$ a solution to \eqref{EDP_KS} in the sense of Definition \ref{6.2.4}, for $q<r<\infty$, we have
$$\mathcal{N}_t^r(\rho)\leqslant C_r.$$
\end{lemma}

\begin{proof}
Let $1<q_1<\frac{d}{d-1}$ and $\frac{d}{2}<q_2<d$ such that $\frac{1}{q_1}=\frac{d-1}{d}+\frac{d-1}{dr}$ and $\frac{1}{q_2}=\frac{1}{d}+\frac{1}{dr}$. Thus, $1+\frac{1}{r}=\frac{1}{q_1}+\frac{1}{q_2}$, and according to Lemma \ref{lemme_normale} we have
$$\|\rho_t\|_r \leqslant \|g_t\ast \rho_0\|_r + \sum_{i=1}^d \int_0^t \|\nabla _i g_{t-s}\|_{q_1} \|\rho_s b^{i}\|_{q_2}ds.$$
With Hölder's inequality for $\frac{1}{\lambda_1}+\frac{1}{\lambda_2}=1$ where $\lambda_1=\frac{q}{d}$, we have $\frac{d}{2}<\lambda_1 q_2 < q$ because $d<q<2d$ and $\lambda_2 q_2 \geqslant d$ because $\lambda_2 >2$. Then,
$$\|\rho_s b^{i}\|_{q_2} \leqslant \|\rho_s\|_{\lambda_1 q_2}\|b^{i}_s\|_{\lambda_2 q_2}.$$
Using Lemma \ref{6.3.9} and Corollary \ref{6.3.10-2} we get
$$\|\rho_s b^{i}\|_{q_2} \leqslant \dfrac{C}{s^{1-\frac{d}{2\lambda_1q_2}+\frac{1}{2}-\frac{d}{2\lambda_2q_2}}}=\dfrac{C}{s^{\frac{3}{2}-\frac{d}{2q_2}}}.$$
Thus, we have finally
$$t^{1-\frac{d}{2r}}\|\rho_t\|_r\leqslant C+t^{1-\frac{d}{2r}}\int_0^t \dfrac{C}{(t-s)^{\frac{d}{2}(1-\frac{1}{q_1})+\frac{1}{2}}s^{\frac{3}{2}-\frac{d}{2q_2}}}ds,$$
To conclude, we use \eqref{6.3.2}.
\end{proof}

\subsection{Global properties}
\label{Ss_properties}

We can now prove Theorem \ref{S_properties_mainthm}.

\begin{proof}[Proof of Theorem \ref{S_properties_mainthm}]
We will proceed by iteration.\\
Let $n\geq 1$, suppose that $n\leq T < n+1$ and suppose that we have, for $r_1\leq \frac{d}{2}$, for $r_2 \geq \frac{d}{2}$, and for $r_3 \geq d$ with $C_r^n$ not depending on $t$, 
$$ \sup_{0 \leq t\leq n} \|\rho_t\|_{r_1} \leqslant C_{r_1}^n, \qquad \mathcal{N}^{r_2}_n(\rho) \leq C^n_{r_2}, \qquad \text{and} \qquad \sup_{0\leq t \leq n} t^{\frac{1}{2}-\frac{d}{2r_3}}\|b_t\|_{r_3}\leq C^n_{r_3}.$$
One can follow the exact same proofs, the only change is to notice that in Step 4 of Lemma \ref{estimation_densite_borne_lemme} we have for all $n\leq t\leq T$, instead of \eqref{estimation_densite_borne}
$$\rho_t = g_{t-n} \ast \rho_n + \sum \int_n^t \nabla_i g_{t-s} \ast (\rho_sb_s)ds$$
and instead of \eqref{estimate_b}
$$b_s = g_{s-n}\ast b_n + \chi \int_n^s \nabla g_{s-u}\ast \rho_u du.$$
Finally one gets instead of \eqref{estimation_densite_borne_2}
$$\|\rho_t\|_{\frac{d}{2}} \leq \|\rho_n\|_{\frac{d}{2}} + d \int_n^t \dfrac{C_1(1)}{\sqrt{t-s}}\dfrac{A_0}{s^{a_0}}\left[\|b_n\|_d + \chi \int_n^s \dfrac{C_1(1)}{\sqrt{t-u}}\dfrac{A_0}{u^{a_0}}du\right]ds.$$
The rest of the proofs is the same.
\end{proof} 

\newpage
\bibliography{biblio}

\end{document}